\documentclass[a4paper,12pt]{scrartcl}

\usepackage[utf8]{inputenc}

\usepackage{amsthm,amsmath,amsfonts,amssymb}
\numberwithin{equation}{section}  
\usepackage{mathrsfs}

\usepackage{tikz-cd}

\usepackage{mathtools}	 

\usepackage{bm}	  
\usepackage{bbm}	 

\usepackage{tensor}	 
\usepackage[colorlinks=true]{hyperref} 
\usepackage[all]{hypcap}       
 
\usepackage{todonotes} 
 
\newcommand{\C}{\mathbb{C}}

\newcommand{\Z}{\mathbb{Z}}
\newcommand{\Q}{\mathbb{Q}}
\newcommand{\ii}{\operatorname{i}}
\newcommand{\eps}{\varepsilon}

\newcommand{\vol}{\operatorname{vol}}

\newcommand{\PP}{\mathbb{P}}

\newcommand{\Bl}{\operatorname{Bl}}

\newcommand{\Aut}{\operatorname{Aut}}
\newcommand{\Bs}{\operatorname{Bs}}
\newcommand{\Ch}{\operatorname{Ch}}
\newcommand{\Ka}{\operatorname{Ka}}

\newcommand{\fd}{\mathfrak{d}}
\newcommand{\fk}{\mathfrak{k}}

\newtheorem{thm}{Theorem}[section]
 
\newtheorem{prop}[thm]{Proposition}
\newtheorem{lemma}[thm]{Lemma}
\newtheorem*{lemma*}{Lemma}
\newtheorem{cor}[thm]{Corollary}

\theoremstyle{definition}
\newtheorem{definition}[thm]{Definition}

\theoremstyle{remark}

\newtheorem{rmk}[thm]{Remark}

\title{Some applications of canonical metrics to Landau-Ginzburg models}
\author{Jacopo Stoppa}
 
\date{\today}

\begin{document}

\maketitle

\begin{abstract} It is known that a given smooth del Pezzo surface or Fano threefold $X$ admits a choice of log Calabi-Yau compactified mirror toric Landau-Ginzburg model (with respect to certain fixed K\"ahler classes and Gorenstein toric degenerations). Here we consider the problem of constructing a corresponding map $\Theta$ from a domain in the complexified K\"ahler cone of $X$ to a well-defined, separated moduli space $\mathfrak{M}$ of polarised manifolds endowed with a canonical metric. We prove a complete result for del Pezzos and a partial result for some special Fano threefolds. The construction uses some fundamental results in the theory of constant scalar curvature K\"ahler metrics. As a consequence $\mathfrak{M}$ parametrises $K$-stable manifolds and the domain of $\Theta$ is endowed with the pullback of a Weil-Petersson form.\\
\textbf{MSC 2010:} 14J33, 32Q26 (Primary) 53C55 (Secondary)  
\end{abstract}
 
\section{Introduction}
In the context of mirror symmetry for Calabi-Yau manifolds, it is expected that the mirror map, in one direction, sends (complexified) K\"ahler classes on $X$ to complex structures on the mirror $X^{\vee}$, where both $X$ and $X^{\vee}$ are compact Calabi-Yau. By a classical result of Schumacher \cite{Schumacher_CY}, there is a \emph{separated} coarse moduli space of such complex structures on $X^{\vee}$ compatible with any fixed K\"ahler class, which moreover carries a natural Weil-Petersson metric. It is interesting to study the pullback of this metric to the complexified K\"ahler cone (see e.g. the discussion by Wilson \cite{Wilson_sectional}, Section 1, and Remark \ref{SYZRmk} below). 

The situation for Fano manifolds is rather different. The mirror to a Fano $X$ is expected to be a Landau-Ginzburg model $Y$, i.e. a noncompact Calabi-Yau manifold endowed with a noncostant holomorphic function $f\!: Y \to \C$. Such objects do not have a good moduli theory in general. As a consequence, \emph{compactified} Landau-Ginzburg models play an important role in mirror symmetry for Fano manifolds, see for example the discussion by Katzarkov, Kontsevich and Pantev in \cite{KatzKontPant}, Section 2.1, and the related results by Cheltsov and Przyjalkowski \cite{CheltPrz_KKP}. Here we concentrate on the notion of \emph{log Calabi-Yau compactification of a toric Landau-Ginzburg model} explained e.g. in \cite{Przyjalkowski_CYLG}, Definitions 6 and 7. 

\subsection{Compactified toric LG models} Fix a pair $(X, H)$ given by a Fano $X$ and an effective divisor $H$. Following Givental \cite{Givental_toric}, the general principle of \emph{Hodge theoretic mirror symmetry} for Fano manifolds states that genus $0$ Gromov-Witten theory on $(X, H)$ should be equivalent to the computation of the classical periods of a regular function $f$ on a noncompact Calabi-Yau manifold $Y$, with respect to a suitable choice of holomorphic volume form $\Omega$. The pair $(f\!: Y \to \C, \Omega)$ is known as a Landau-Ginzburg model (LG model). 

The notion of toric LG model, recalled below, selects a particular class of genus $0$ Gromov-Witten invariants to turn this principle into a definition. As explained in \cite{Coates_quantum}, Section B, these are the \emph{$1$-pointed, genus $0$ Gromov-Witten invariants with descendent}
\begin{equation*}
\langle \beta \psi^a \rangle_{0,1,\beta} := \int_{[X_{0,1,\beta}]^{\operatorname{vir}}} \operatorname{ev}^*(\beta) \cup \psi^a \in \Q,
\end{equation*}
where $X_{0,1,\beta}$ denotes the relevant moduli space of stable maps with $1$ marked point with degree $\beta \in H^*(X, \Z)$, with virtual fundamental class $[X_{0,1,\beta}]^{\operatorname{vir}}$, $\operatorname{ev}$ denotes evaluation at the image of the marked point, and $\psi \in H^2(X_{0,1,\beta})$ denotes the first Chern class of the universal cotangent line bundle over $X_{0,1,\beta}$. 
\begin{definition}[\cite{Przyjalkowski_CYLG}, Definitions 6 and 7]\label{CompactLGDef} A \emph{log Calabi-Yau compactification of a toric LG model} for $(X, H)$ is a projective morphism $f\!:Z\to \PP^1$, satisfying the following conditions
\begin{itemize}
\item[$(i)$] \emph{log Calabi-Yau condition:} $Z$ is smooth and we have $-K_Z \sim f^{-1}(\infty)$;
\item[$(ii)$] \emph{period condition:} setting $Y:= Z \setminus f^{-1}(\infty)$, there is a dense open torus $(\C^*)^n \subset Y$ on which $f$ restricts to a \emph{toric LG model} for $(X, H)$, that is, a Laurent polynomial such that its \emph{classical period},    
\begin{equation*}
I_f(t) := \left(\frac{1}{2\pi \ii}\right)^n\int_{\Gamma} \frac{\Omega}{1-t f}   
\end{equation*}
equals the \emph{constant term of the regularized $I$-series for $(X, H)$ in genus $0$} (also known as the \emph{regularised quantum period} of $(X, H)$, see \cite{Coates_quantum}, Section B),
\begin{equation*}
\widetilde{I}^{X, -K_X, H}_0(t) := 1 + \sum_{\beta \in H_2(X, \Z)} (-K_X\cdot\beta)! \langle [\operatorname{pt}] \psi^{-K_X\cdot \beta -2}\rangle_{0,1, \beta} e^{-\beta\cdot H} t^{-K_X\cdot\beta},
\end{equation*}
namely we have the identity of convergent formal power series in a neighbourhood of the origin
\begin{equation*}
I_f(t) = \widetilde{I}^{X, -K_X, H}_0(t). 
\end{equation*}   
Here, the standard holomorphic volume form on $(\C^*)^n$ is given by 
\begin{equation*}
\Omega = \left(\frac{1}{2\pi \ii}\right)^n \frac{dx_1}{x_1} \wedge \cdots \wedge \frac{dx_n}{x_n}, 
\end{equation*}
and the integration cycle is
\begin{equation*}
\Gamma = \{|x_1| = \cdots = |x_n| = 1\};
\end{equation*}
\item[$(iii)$] \emph{toric condition:} the toric LG model $f$ appearing in $(ii)$ arises from a corresponding toric degeneration of $X$, i.e. the Fano manifold $X$ specialises to a toric variety $T_X$ such that the fan polytope $F(T_X)$ equals the Newton polytope of $f$. 
\end{itemize}
\end{definition}
Toric LG models, in themselves, have been studied in depth since the classical work of Givental \cite{Givental_toric} by several authors, with fundamental contributions by Coates, Corti, Galkin, Golyshev, Kasprzyk (see e.g. \cite{CoatesCorti_ECM}).
\begin{rmk} The explicit form of $\widetilde{I}^{X, -K_X, H}_0(t)$ shows that it corresponds to the ``polarisation" (big divisor) $-K_X + H$. It is standard to extend the above definition to allow a complexified effective divisor $H$ (i.e. a $\C$-divisor with effective real part) and to regard $-K_X + H$ as a generalisation of a complexified K\"ahler class $[\omega_{\C}]$.
\end{rmk}
\begin{rmk} Katzarkov, Kontsevich and Pantev (\cite{KatzKontPant}, Section 2.1) explain how the homological mirror symmetry approach extends (conjecturally) to Fano manifolds. In particular, if $(f\!: Y \to \C, \Omega)$ is mirror to a Fano $X$ with complexified K\"ahler form $\omega_{\C}$, then there should be an equivalence
\begin{equation*}
\operatorname{Fuk}(X, \omega_{\C}) \cong \operatorname{MF}(Y, f) 
\end{equation*}
between the Fukaya category of $(X, \omega_{\C})$ (a symplectic invariant) and the category of matrix factorisations of $(f\!: Y \to \C, \Omega)$ (a holomorphic invariant), see \cite{KatzKontPant}, Table 1. In the other direction, one expects an equivalence  
\begin{equation*}
D^b(X) \cong \operatorname{FS}(f\!: Y \to \C, (\omega_{Y})_{\C}, \Omega)
\end{equation*}
between the bounded derived category of $X$ (a holomorphic invariant) and the Fukaya-Seidel category of $(f\!: Y \to \C, \Omega)$ with respect to a suitable choice of complexified K\"ahler form $(\omega_{Y})_{\C}$ (a symplectic invariant). 

The work \cite{KatzKontPant} explains how the homological mirror symmetry approach leads to the study of a natural class of log Calabi-Yau compactifications of $(f\!: Y \to \C, \Omega)$, called tame compactifications, for which homological mirror symmetry gives interesting predictions: for example, they should have unobstructed deformations, as well as intrinsic Hodge numbers, matching the Hodge numbers of $X$ (the latter prediction was confirmed for threefolds in \cite{CheltPrz_KKP}).
\end{rmk}

By the results of Przyjalkowski \cite{Przyjalkowski_CYLG}, Section 3, when $X$ is a del Pezzo surface, then there exists an \emph{explicit, canonical} construction of a compactified toric LG model, in the sense above, for all pairs $(X, H)$.

When $X$ is a Fano threefold, on the other hand, Przyjalkowski (\cite{Przyjalkowski_CYLG}, Theorem 1), using his recently appeared extensive work with Doran, Harder, Katzarkov, Ovcharenko \cite{DoranEtAl_modularity}, shows that the pair $(X, 0)$, for fixed Laurent polynomial $f$, always admits a log Calabi-Yau compactified toric LG model $Z$. This is not unique, but the corresponding open Calabi-Yaus $Y$ are isomorphic away from closed subsets of codimension two. 

\subsection{Maps of moduli spaces}
Suppose $X$ is a Fano manifold with a fixed toric degeneration. In the present work we consider the corresponding ``framed mirror map" $\Theta$ taking a complexified K\"ahler class $[\omega_{\C}] \in \Ka(X)_{\C}$ to (a choice of) its compactified toric LG model $Z$ (in cases when this is well defined). We study the problem of realising $\Theta$ as a map taking ``(complexified) K\"ahler moduli" to ``complex moduli" in a precise sense, that is, of constructing $\Theta$ as a holomorphic map 
\begin{equation}\label{MirrorMapModuli}
\Theta\!: \Ka(X)_{\C} \supset U \to \mathfrak{M}_U 
\end{equation}
from a \emph{large, connected open neighbourhood $U$ of the anticanonical class of $X$} (in the analytic topology) to a \emph{separated} coarse moduli space $\mathfrak{M}_U$, possibly depending on $U$, parametrising \emph{stable polarised manifolds} (recall stability is essential to obtain the Hausdorff property for $\mathfrak{M}$), and known to carry a natural Weil-Petersson metric. We focus on the case when $\mathfrak{M}_U$ is a moduli space parametrising K\"ahler manifolds admitting certain \emph{canonical metrics}. We will see that in particular these are \emph{$K$-stable polarised manifolds}.

Our approach relies on the theory of constant scalar curvature K\"ahler metrics (\emph{cscK metrics}). Suppose we can show that there exists a holomorphic family of log Calabi-Yau compactified mirrors parametrised by $U$ such that all members $Z$ have discrete automorphisms and that we can choose a \emph{fixed} class $[\eta] \in \Ka(Z)$ such that, as the complex structure on $Z$ varies holomorphically with $[\omega_{\C}] \in U$, $[\eta]$ always admits a (necessarily unique, up to automorphisms \cite{Donaldson_embeddings, BermanDarvasLu_weakminimizers}) cscK representative. Then the existence of $\Theta$ as a map of moduli spaces, as above, follows from the following classical result.
\begin{thm}[Fujiki-Schumacher \cite{FujikiSchumacher_moduli}] There exists a separated complex space $\mathfrak{M}$ which is a coarse moduli space for compact K\"ahler manifolds with discrete automorphisms, endowed with a cscK metric representing a fixed cohomology class (on the underlying smooth manifold). Moreover, $\mathfrak{M}$ carries a canonical Weil-Petersson form $\Omega_{WP}$.  
\end{thm}
\begin{rmk}Although we will not use this here, we note that this result has been reproved and extended to allow continuous automorphisms by Dervan and Naumann \cite{DervanNaumann_moduli}. On the other hand an analogous result for \emph{extremal} metrics (i.e. critical points of the $L^2$ norm of the scalar curvature) has not been proved yet.  
\end{rmk}
By the results of \cite{Donaldson_lower_bounds, Stoppa_cscK_stabili} on $K$-semistability and discrete automorphisms, the K\"ahler manifolds appearing in the Fujiki-Schumacher Theorem are in fact $K$-stable (see \cite{BermanDarvasLu_weakminimizers} for the case of continuous automorphisms).
\subsection{del Pezzo surfaces}\label{delPezzoSecIntro}
We first prove a rather complete result in the case of smooth del Pezzo surfaces. We will also prove a similar (weaker) result for some special Fano threefolds (Theorem \ref{MainThmFanos}).
\begin{thm}\label{MainThmDelPezzo} Suppose $X$ is a smooth del Pezzo surface with a fixed Gorenstein toric del Pezzo degeneration. Then it is possible to construct the corresponding mirror map $\Theta$ at the level of moduli spaces as in \eqref{MirrorMapModuli}, where $U \subset \Ka(X)_{\C}$ is an explicit, relatively compact but arbitrarily large open neighbourhood of $c_1(X)$ and $\mathfrak{M}_U$ is a Hausdorff coarse moduli space of cscK (hence $K$-stable) polarised rational elliptic surfaces. In particular, the domain $U$ carries the $(1,1)$-form $\Theta^* \Omega_{WP}$, where $\Omega_{WP}$ is the canonical Weil-Petersson metric on $\mathfrak{M}_U$.
\end{thm}
The proof is carried out in Sections \ref{PP2Sec}-\ref{7aSec}, following the outline we provide in the present Section.
\begin{rmk}\label{twistedCsckRemark} Once we construct the map $\Theta$ as above, we can keep track of the map $f$ as induced by a section of the sheaf over $\mathfrak{M}_U$ with fibre over $Z$ given by $H^0(Z, -K_Z)$ (indeed, as we recall below, the map to $\PP^1$ is given by a canonical extension of the corresponding rational pencil).

It is also natural to consider the map $\widetilde{\Theta}$ taking $[\omega_{\C}]$ to the morphism $f\!: Z \to \PP^1$, rather than to the total space $Z$. From the point of view of canonical metrics there is a natural extension of the cscK equation uniformising such maps, namely
\begin{equation}\label{twistedCscK}
s(\eta) - \varepsilon \Lambda_{\eta} f^* \omega_{FS} = c
\end{equation}
where $[\eta]$ is a K\"ahler class on $Z$, $s(\eta)$ denotes the scalar curvature, $\varepsilon > 0$ is a parameter and $c$ is a uniquely determined constant, see the work of Dervan-Ross \cite{DervanRoss_stablemaps}. Since in our case $Z$ has discrete automorphisms, standard arguments in elliptic theory show that in the setup of Theorem \ref{MainThmDelPezzo} the equation \eqref{twistedCscK} is also solvable on $Z$ for sufficiently small $\varepsilon > 0$, and moreover that $\varepsilon$ can be chosen uniformly on $U$. As a consequence, we do obtain a corresponding map to the \emph{set} of solutions of \eqref{twistedCscK} for fixed $[\eta]$ and $\varepsilon > 0$. However, the existence of a separated coarse moduli space $\widetilde{\mathfrak{M}}$ for such solutions is not yet known, although it is expected. Thus, modulo this conjectural existence result, Theorem \ref{MainThmDelPezzo} would also provide a map $\widetilde{\Theta}\!: U \to \widetilde{\mathfrak{M}}_U$.
\end{rmk}
\begin{rmk}\label{SYZRmk} By analogy with the Calabi-Yau case discussed in \cite{Wilson_sectional}, Section 1, one could ask whether, as the parameters $[\omega_{\C}]$  and $[\omega_Z]$ approach some appropriate limiting values, the corresponding pullback form $\Theta^* \Omega_{WP}$ is approaching the $L^2$ metric on $\Ka(X)_{\C}$ induced by a sequence of degenerating K\"ahler structures on $X$. Note that in fact the classical work of Auroux \cite{Auroux_complement} does not assume completeness or Ricci flatness, that is, the restriction of $\omega_Z$ to $Y = Z \setminus f^{-1}(\infty)$ would give an admissible choice of K\"ahler form. Similarly, the well-known computations of Leung \cite{Leung_woCorrections}, which prove the analogous statement in the semi-flat case, do not actually depend on Ricci flatness (rather, Leung shows that his construction is always compatible with the real Monge-Amp\`ere equation). 
\end{rmk}
Let us outline the basic ingredients for the proof of Theorem \ref{MainThmDelPezzo}. The fixed Gorenstein toric degeneration $T_X = T_{\Delta}$ of $X$, where $\Delta$ is a reflexive Fano polytope, provides a toric LG model. There is a corresponding Gorenstein dual toric variety $T_{\nabla}$, such that $\nabla$ is the polar dual of $\Delta$, with an anticanonical rational pencil $f\!: T_{\nabla} \dashrightarrow \PP^1$. Following \cite{Przyjalkowski_CYLG}, Section 3, the canonical log Calabi-Yau compactification $Z$ of $X$ is obtained by constructing a crepant resolution of the (orbifold) singularities of $T_{\nabla}$, as well as of the base locus $\Bs(f)$ of $f$ (one shows that the two sets are disjoint). Using explicit formulae for the mirror map, one sees that for $[\omega_{\C}]$ lying in a dense open locus, $\Bs(f)$ is reduced and varies holomorphically with $[\omega_{\C}]$. Then we can use a combination of results on K\"ahler-Einstein (orbifold) metrics \cite{WangZhu_toric, Zhu_orbisolitons} and cscK metrics on blowups and resolutions of orbifold singularities (Arezzo-Pacard type theorems such as \cite{ArezzoPacard, ArezzoPacardSinger, Gabor_blowup, Arezzo_kummer}) in order to construct a family of cscK classes on $Z$. We emphasise that this holds even in cases when $T_{\nabla}$ does \emph{not} admit an orbifold K\"ahler-Einstein metric, by choosing the order of metric blowups (i.e. the corresponding K\"ahler parameters) suitably. This family of cscK polarisations depends on $[\omega_{\C}] \in \Ka(X)_{\C}$, but by analysing the estimates appearing in the proofs of these results, one can see that it is possible to pick a single cscK class $[\eta]$ which works for all $[\omega_{\C}]$ as long as the distances between the base points and between the base locus and the torus-fixed points of $T_{\nabla}$ are uniformly bounded away from zero. 

However, this generic result is not enough for our purposes, since there are strata in $\Ka(X)_{\C}$ such that the base locus $\Bs(f)$ is nonreduced. In many cases, these are the most interesting strata, containing for example the anticanonical class $c_1(X)$. By construction, the corresponding complex structure on $Z$ is given by an iterated blowup. We show that it is possible to construct cscK classes on such blowups following the strategy above, and that moreover neighbourhoods of such strata can be glued to the generic locus, by using the deformation theory of cscK metrics (\cite{FujikiSchumacher_moduli, LeBrunSimanca, Szekelyhidi_moduli}). Again, uniform estimates are essential. 

\begin{rmk}\label{AdiabaticRmk} The compactified LG model $Z$ is always a rational elliptic surface. The $K$-(in)stability of rational elliptic surfaces has been studied by Hattori \cite{Hattori_CYfibrations} in the case of adiabatic classes $H + f^*c_1(\mathcal{O}_{\PP^1}(k))$, for $k \gg 1$. These are very far from the classes we use in our proof of Theorem \ref{MainThmDelPezzo}. It would be interesting to see whether Hattori's results can be used to prove an analogue of Theorem \ref{MainThmDelPezzo} for different polarisations; the main question is whether the adiabatic parameter $k$ can be chosen uniformly. This question can be studied using the recent results of Hashizume and Hattori \cite{Hattori_moduli}. However the following observation shows that this approach cannot work in complete generality.

\begin{lemma*}[Section \ref{AdiabaticSec}] There exist adiabatically $K$-unstable compactified toric LG models of smooth del Pezzo surfaces.
\end{lemma*}     

\end{rmk}

Although the strategy of the proof of Theorem \ref{MainThmDelPezzo} is the one outlined above, the actual proof requires a certain case by case analysis. The approach we follow in our presentation is to give full details for the cases which illustrate all the relevant features: in particular, when the surface $X$ is given by $\PP^2$, $\Bl_p \PP^2$ or a del Pezzo of degrees 3, 4, 6, endowed with certain Gorenstein toric degenerations (see Sections \ref{PP2Sec}-\ref{P8aSec}). The other cases are treated much more briefly (Sections \ref{QuadraticConeSec}-\ref{7aSec}). For our case by case analysis of toric mirrors of del Pezzos the works \cite{Petracci_reflexive, Przyjalkowski_CYLG} provide excellent references. In particular we will follow the numeration of reflexive del Pezzo polygons in \cite{Petracci_reflexive}, Fig. 2 (while in dimension 3 we follow the standard Kreuzer-Skarke list \cite{KreuzSkarke}).

\subsection{HMS pairs}
One advantage of the approach using canonical K\"ahler metrics is that, in the case of discrete automorphisms, their deformation theory is unostructed: a small deformation of complex structure of a cscK manifold is still cscK (in the same polarisation for the underlying smooth manifold), see \cite{FujikiSchumacher_moduli, LeBrunSimanca, Szekelyhidi_moduli}. This lends some flexibility to our construction. We provide an application of this basic fact to mirror pairs in the sense of homological mirror symmetry (HMS). 

Suppose $Z$ is a compactified toric LG model endowed with a cscK metric $\omega_Z$, as constructed above. Let $\mathfrak{M}$ denote the Fujiki-Schumacher (Hausdorff) moduli space of cscK complex structures compatible with $\omega_Z$ (in particular such polarised manifolds are $K$-stable). The pair $(Z, D := f^{-1}(\infty))$ is a log Calabi-Yau surface. Gross-Hacking-Keel (see \cite{GrossHackingKeel_LCY}, Conjecture 0.13) construct a family of del Pezzo surfaces with \emph{smooth} anticanonical divisors 
\begin{equation*}
(\bar{\mathcal{X}}, \mathcal{D}) \to S = \{|\exp(2\pi\ii\int_D(B + \ii \omega))| < 1,\,\forall\text{ effective } D\}\subset \operatorname{Pic}(Z)\otimes\C^*
\end{equation*} 
($[B + \ii \omega] \in \operatorname{Pic}(Z)\otimes\C^*$ denoting a point in the complexified K\"ahler cone) such that, for $s \in S$, the complex structure on $\bar{\mathcal{X}}$ is mirror to the K\"ahler form on $Z$ in the sense of homological mirror symmetry, that is, we have an equivalence 
\begin{equation*}
D^b(\bar{\mathcal{X}}_s) \cong \operatorname{FS}(Z \setminus D, \ii \omega|_{Z\setminus D})
\end{equation*}
between the bounded derived category of the category of coherent sheaves on $\bar{\mathcal{X}}_s$ and the Fukaya-Seidel category of $(Z \setminus D, \ii \omega|_{Z\setminus D})$ (this claim was conditional on certain results later proved by Lai-Zhou \cite{LaiZhou_mirror}). Since the cscK condition is scale invariant, we can apply this result to $Z$ up to replacing $[\omega_Z]$ with a suitable scaling (to ensure $|s| < 1$). Conversely, under the mirror map for pairs $(\bar{\mathcal{X}}, \mathcal{D})$, the fixed complex structure on $Z$ corresponds to a K\"ahler class $[\omega_{\bar{\mathcal{X}}_s}]$.
\begin{cor} There is a nonempty open neighbourhood of $[\omega_{\bar{\mathcal{X}}_s}]$ in the complexified K\"ahler cone of $\bar{\mathcal{X}}_{s}$ which maps to the moduli space $\mathfrak{M}$ under the mirror map for pairs $(\bar{\mathcal{X}}, \mathcal{D})$.
\end{cor}
This follows at once from the fact that $Z$ has discrete automorphisms and from the results about deforming cscK complex structures. 
\subsection{Fano threefolds}\label{3foldsSecIntro}
Proving an analogue of Theorem \ref{MainThmDelPezzo} in the case of smooth Fano threefolds is much more challenging of course. Here we confine ourselves to a study of a few special Fano threefolds $X$. At present, we are only able to obtain a result for a sufficiently small domain $U \ni -K_X$, which is moreover conditional on some conjectural expectations. 
\begin{itemize}
  
\item[$(a)$] For each smooth Fano threefold $X$, the Fanosearch  programme \cite{Fanosearch} provides several Laurent polynomials $f_X$, such that $f_X$ satisfies the period condition (Definition \ref{CompactLGDef}, $(ii)$) with respect to $-K_X$.

\emph{In our applications, given $X$ in our special class, we will make a specific choice of such $f_X$ and assume that it actually arises from a toric degeneration of $X$, i.e. that the condition $(iii)$ in Definition \ref{CompactLGDef} is satisfied.}
\item[$(b)$] Fix $X$, $f_X$ as in $(a)$. In our applications, we will fix a log Calabi-Yau compactification $f\!: Z \to \PP^1$ of $f_X$ (constructed as in \cite{Przyjalkowski_CYLG}, Section 4).
 
\emph{We assume that there is a well defined holomorphic map, defined in a neighbourhood $-K_X \in  U \subset \Ka(X)_{\C}$, which associates to a class $ -K_X + H  \in U$ a log Calabi-Yau compactified LG model for $(X, H)$, obtained as a deformation of $f\!: Z \to \PP^1$}. 

This is the natural \emph{modularity condition} for a Landau-Ginzburg model, saying that it deforms along with the K\"ahler parameters of $X$. As we mentioned after Definition \ref{CompactLGDef}, it is known for del Pezzos, and it is in fact expected for Fano threefolds according to \cite{DoranEtAl_modularity}, Proposition 1.7  and the conjecture discussed in \cite{DoranEtAl_modularity}, Section 1.4 (partially proved there in Proposition 1.11 and Appendix K).

\item[$(c)$] Seyyedali and Sz\'{e}kelyhidi \cite{GaborSeyyedali} prove the analogue of the Arezzo-Pacard criterion for blowing up along submanifolds of codimension greater than $2$, and conjecture that the same result holds for codimension $2$. 

\emph{In our applications, we assume that this holds, i.e. that \cite{GaborSeyyedali}, Theorem 1 also holds for submanifolds of codimension $2$ of the compactification $Z$ chosen in $(b)$.}

\end{itemize}

\begin{thm}\label{MainThmFanos} Suppose $X$ is a Fano threefold of type $V_4$, $V_6$, $V_8$, $V_{10}$, $V_{12}$, a $(2,2)$ divisor in $\PP^2 \times \PP^2$, a $(2,2,2)$ branched cover of $(\PP^1)^3$ or a $(1,1,1,1)$ divisor in $(\PP^1)^4$. Assume that the conjectures $(a)$, $(b)$, $(c)$ above hold for $X$. Then, there exist Gorenstein toric degenerations of $X$ such that the corresponding mirror map $\Theta$ can be constructed at the level of moduli spaces as in \eqref{MirrorMapModuli}, where $U \subset \Ka(X)_{\C}$ is an open neighbourhood of $c_1(X)$, and $\mathfrak{M}$ is a Hausdorff coarse moduli space of cscK (hence $K$-stable) polarised manifolds. In particular, $U$ carries the $(1,1)$-form $\Theta^* \Omega_{WP}$.
\end{thm}
\begin{rmk}\label{3foldsAdiabaticRmk} Our observations in Remarks \ref{twistedCsckRemark} (concerning moduli of maps) and \ref{AdiabaticRmk} (on adiabatic stability) also apply to this result (the latter with reference to the higher dimensional results proved in \cite{Hattori_moduli}, Theorem 1.1). The Lemma stated in Remark \ref{AdiabaticRmk} also generalises as follows.
\begin{lemma*}[Section \ref{AdiabaticSecQuartic}] The LG model of the quartic threefold $V_4$ appearing in Theorem \ref{MainThmFanos} is adiabatically $K$-unstable. 
\end{lemma*}
\end{rmk}

Note that $V_4$, $V_6$, $V_8$, $V_{10}$, $V_{12}$ have Picard rank $1$, so the corresponding $U$ is a disc, while the $(2,2)$ divisor, the $(2,2,2)$ branched cover and the $(1,1,1,1)$ divisor have Picard ranks $2$, $3$ and $4$ respectively.

Theorem \ref{MainThmFanos} is proved (conditionally) in Sections \ref{Quartic3foldSec}-\ref{1111Sec} by showing that, for all $X$ as in the statement except $V_{10}$ and the $(1,1,1,1)$ divisor, there is a (potential) Gorenstein toric degeneration $T_{\Delta}$ such that the dual variety $T_{\nabla}$ is smooth K\"ahler-Einstein and the base locus of the Minkowski Laurent polynomial $f$ corresponding to $-K_X$ can be resolved by repeated blowups of polystable cycles under the induced linearised action of (subgroups of) $\Aut_0(T_{\nabla})$, so that after a full resolution the residual automorphisms are discrete. The (partially conjectural) results of \cite{DoranEtAl_modularity} provide a mirror map on a neighbourhood $U$, and since the automorphisms are discrete, by the deformation theory of cscK metrics, up to shrinking $U$ the underlying deformations of $Z$ are still cscK. The case of the $(1,1,1,1)$ divisor is similar, except $T_{\nabla}$ is orbifold K\"ahler-Einstein with isolated singularities.

The assumption that $T_{\nabla}$ is (orbifold) K\"ahler-Einstein is much too restrictive (this is clear by comparing to the case of del Pezzo surfaces), and we illustrate this in the missing case $X = V_{10}$. In this case, $X$ has a Gorenstein toric degeneration such that the dual variety $T_{\nabla}$ is smooth but $\Aut_0(T_{\nabla})$ is not reductive (and so a fortiori $T_{\nabla}$ is not even cscK). Nevertheless, $T_{\nabla}$ is isomorphic to a blowup $\Bl_{E} T_{\nabla'}$ where $T_{\nabla'}$ is K\"ahler-Einstein, and we can still apply the Arezzo-Pacard Theorem by blowing up $E$ and the components of $\Bs(f)$ in a suitable order. \\ 

\noindent{\textbf{Acknowledgements.}} I am grateful to Giulio Codogni, Ruadha\'i Dervan, Annamaria Ortu, Carlo Scarpa and Roberto Svaldi for some helpful discussions. This research was carried out in the framework of the project PRIN 2022BTA242 ``Geometry of algebraic structures: moduli, invariants, deformations". The author would like to thank the Isaac Newton Institute for Mathematical Sciences, Cambridge, for support and hospitality during the programme New equivariant methods in algebraic and differential geometry, where work on this paper was undertaken. This work was supported by EPSRC grant EP/R014604/1.
\section{$X = T_{P_3} = \PP^2$}\label{PP2Sec}
In this Section we start the case by case analysis which is necessary for the proof of Theorem \ref{MainThmDelPezzo}, as outlined in Section \ref{delPezzoSecIntro}. This is completed in Section \ref{7aSec}.

We consider the case $X = \PP^2$, following the notation of \cite{Przyjalkowski_CYLG}, Section 3. Suppose $H = a_0 l$ for $a_0 \geq 0$, where $l$ denotes the class of a line. By classical work of Givental \cite{Givental_toric}, a toric LG model for $X$ is given by
\begin{equation*}
f_{\PP^2, H} = x + y + e^{-a_0}\frac{1}{x y}. 
\end{equation*}

\subsection{First step: dual polytope} The fan polytope of $X$ is given by the reflexive polygon $\Delta = P_3$ in \cite{Petracci_reflexive}, Fig. 2. Its polar dual $\nabla$ is given by the reflexive polygon $\nabla = (P_3)^{\circ} = P_9$ in the same list. The canonical construction of the log Calabi-Yau compactified toric LG model for $X$ starts by considering the singular toric variety $T_{\nabla}$. Since $\nabla$ is reflexive, $T_{\nabla}$ is Gorenstein. 

By the general theory recalled in \cite{Przyjalkowski_CYLG}, Section 2, Facts 1.-4., $T_{\nabla}$ is given by the \emph{anticanonically embedded} singular cubic surface 
\begin{equation*}
T_{\nabla} = \{x_1 x_2 x_3 = x^3_0\} \subset \PP := \PP[x_0:x_1:x_2:x_3],
\end{equation*}  
(i.e. $-K_{T_{\nabla}} = \mathcal{O}_{\PP}(1)|_{T_{\nabla}}$), with orbifold singularities at the torus fixed points. The Laurent polynomial $f_{\PP^2, H}$ is the restriction to the structure torus of $T_{\nabla}$ of the birational map $T_{\nabla} \dashrightarrow \PP^1$ defined by the pencil
\begin{equation*}
\fd_{\Delta} := |x_1 + x_2 + e^{-a_0}x_3, x_0|.
\end{equation*} 
The divisor $(x_0)$ is precisely the toric boundary of $T_{\nabla}$.

This step can be performed \emph{metrically}. The barycenter of $\Delta$ is the origin, corresponding of course to the Fubini-Study metric $\omega_{\Delta}$ on $X$ under the existence theorem of \cite{WangZhu_toric}. Moreover, $\Delta$ is in fact \emph{symmetric}, in the sense that the group $\Aut_{\Delta}(\Z^2)$ of linear lattice isomorphisms leaving $\Delta$ invariant fixes only the origin. According to \cite{Nill_reductive}, Remark 5.5, the polar dual of a symmetric reflexive polytope is also symmetric, which implies that its barycenter must vanish. So, the barycenter of $\nabla$ must vanish. Naturally, this can be readily checked by hand in this case. By \cite{Zhu_orbisolitons}, Theorem 1.4, Remark 1.5 and Proposition 3.2, it follows that $T_{\nabla}$ carries a \emph{unique orbifold K\"ahler-Einstein metric} $\omega_{\nabla}$ (modulo automorphisms).  

\subsection{Second step: resolving the base locus} 
Recall that on $T_{\nabla}$ we have a pencil $\fd_{\Delta}$ which extends the Laurent polynomial $f_{\PP^2, H}$ to a rational map $T_{\nabla} \dashrightarrow \PP^1$. By our previous discussion, the base locus of $\fd_{\Delta}$ is contained in the toric boundary and is cut out by the equations in $\PP[x_0:x_1:x_2:x_3]$
\begin{equation*} 
\begin{dcases}
x_1 x_2 x_3 = 0\\ 
x_1 + x_2 + e^{-a_0} x_3 = 0
\end{dcases}
\end{equation*} 
and so we have
\begin{equation*}
\Bs(\fd_{\Delta}) = \{p_1, p_2, p_3\} := \{[0:0:1:-e^{ a_0}],\,[0:1:0:-e^{ a_0}],\,[0:1:-1:0]\}.
\end{equation*}
In particular, $\Bs(\fd_{\Delta})$ is disjoint from the torus fixed (and singular) points of $T_{\nabla}$. Set $Z' = \Bl_{\Bs( \fd_{\Delta})} T_{\nabla}$. Then, the pencil $\fd_{\Delta}$ induces a regular morphism $f\!: Z' \to \PP^1$. 

We would like to perform this step \emph{metrically}, i.e. to endow the blowup $p\!: Z'\to T_{\nabla}$ with \emph{natural, canonical metrics $\omega_{Z', \eps}$ obtained from the orbifold K\"ahler-Einstein metric $\omega_{\nabla}$}. The problem of blowing up configurations of smooth points, in a metric way, in the presence of continuous automorphisms, is very well studied, see e.g. \cite{ArezzoPacardSinger, Gabor_blowup}.

Let $\mu\!: T_{\nabla} \to \fk^{\vee}$ denote a moment map for the action of the group $K$ of Hamiltonian isometries of $(T_{\nabla}, \omega_{\nabla})$, with Lie algebra $\fk$, normalised by $\int \langle \mu, v\rangle \omega^2_{\nabla} = 0$ for all $v \in \fk$. 
\begin{rmk} The extension of K\"ahler geometry to orbifolds (such as $T_{\nabla}$) is discussed in several references including \cite{DingTian}, Section 1, \cite{Legendre_localised}, Section 3, \cite{RossThomasOrbi}, Section 1 and \cite{Zhu_orbisolitons}, Section 2.
\end{rmk}
By a special case of \cite{Gabor_blowup}, Theorem 1, if the \emph{balancing condition}
\begin{equation*}
\sum^3_{i = 1} \mu(p_1) = 0
\end{equation*} 
holds, then each K\"ahler class  
\begin{equation*}
p^*[\omega_{\nabla}] -\eps^2 \sum_{i = 1} [E_i],\, \eps \in (0, \bar{\eps})\text{ for some } \bar{\eps} > 0,
\end{equation*}
where $E_i \subset Z'$ denote the exceptional divisors of $p$, admits an extremal metric $\omega_{Z', \eps}$. But since the reduced automorphism group (i.e. the connected component of the identity in the group of Hamiltonian automorphisms) of $Z'$ is trivial (since $p\!: Z' \to T_{\nabla}$ blows up a point in each irreducible component of the toric boundary of $T_{\nabla}$, and there are no nontrivial holomorphic vector fields on $T_{\nabla}$ vanishing at all these blowup centres), it follows that this extremal metric is in fact the \emph{unique cscK metric $\omega_{Z',\eps}$ representing its K\"ahler class}. 
\begin{rmk} Strictly speaking, the results of \cite{Gabor_blowup} are stated for points on a manifold, not smooth points on an orbifold. However, by analysing the proofs, one can check that the results still hold in this case, and that in fact $\bar{\eps} > 0$ can be chosen uniformly with respect to $\Bs(\fd_{\Delta})$ as long as it is uniformly bounded away from the singular locus. By construction, this is the case when $-K_X + H = (3 + a_0) l$ lies in a bounded neighbourhood of $-K_X$ in the K\"ahler cone of $X$.   
\end{rmk}
More generally, suppose there exists $g \in G = K^{\C}$ such that the balancing condition holds for the set $g \cdot \Bs(\fd_{\Delta})$, namely we have
\begin{equation*}
\sum^3_{i = 1} \mu(g\cdot p_1) = 0.
\end{equation*} 
Then, the blowup at $g \cdot \Bs(\fd_{\Delta})$ is cscK in the relevant classes. But since the two blowups are isomorphic, we see that $Z'$ is also cscK in the classes above.

Let us check the balancing condition in our case. We give two different approaches.
\subsubsection{Global quotient}
The orbifold $T_{\nabla}$ can be presented as a global quotient $T_{\nabla} = \PP^2/(\Z/(3))$, where the action of $\Z/(3)$ is by 
\begin{equation*}
\zeta\cdot[z_0:z_1:z_2] = [z_0: \zeta z_1:\zeta^2 z_2],
\end{equation*}
for $\zeta$ a primitive root of unity. The above embedding of $T_{\nabla}$ into $\PP[x_0:x_1:x_2:x_3]$ is induced by the $\Z/(3)$-invariant morphism
\begin{equation*}
[z_0: z_1: z_2] \mapsto [z_0 z_1 z_2: z^3_0:z^3_1:z^3_2].
\end{equation*}
Under this identification we have 
\begin{equation*}
\Bs(\fd_{\Delta}) = \{p_1, p_2, p_3\} := \{[0:1: \xi e^{a_0/3}],\,[1:0: \xi e^{a_0/3}],\,[1:\xi:0 ]\},\,\xi^3 = -1.
\end{equation*}  
The group $G = K^{\C}$ is given by the structure torus acting on $T_{\nabla}$ as
\begin{equation*}
(\C^*)^2 \ni g = (\lambda, \nu )\cdot[z_0:z_1:z_2] = [\lambda z_0 : \lambda^{-1} \nu z_1:\nu^{-1} z_2].
\end{equation*}
The orbifold K\"ahler-Einstein metric $\omega_{\nabla}$ on $T_{\nabla}$ is induced from the Fubini-Study metric on $\PP^2$. This makes it possible to compute the moment map $\mu$ with respect to $\omega_{\nabla}$ explicitly, as in \cite{Arezzo_kummer}, Section 7.4. Set
\begin{align*}
& \phi_1 = \frac{1}{3} - \frac{|z_0|^2}{|z_0|^2+|z_1|^2+|z_2|^2},\,\phi_2 = \frac{1}{3} - \frac{|z_2|^2}{|z_0|^2+|z_1|^2+|z_2|^2},\\
& \varphi_1 = -3(\phi_1+2\phi_2),\,\varphi_2 = -3(2\phi_1+\phi_2).
\end{align*}
Then, we have
\begin{equation*}
\mu(g\cdot p_j) = (\varphi_1(g\cdot p_j), \varphi_2(g\cdot p_j))^{T},
\end{equation*}
so the balancing condition $\sum^3_{j=1}\mu(g\cdot p_j) = 0$ holds iff the kernel of the matrix $[\varphi_i(g\cdot p_j)]$ contains the vector $(1,1,1)^{T}$. We compute explicitly the components
\begin{align*}
&([\varphi_i(g\cdot p_j)](1,1,1)^{T})_1\\
& = \frac{3 \left(e^{2 a_0 /3} \nu ^2 \left(|\lambda|^8-|\nu|^4\right)+e^{4 a_0 /3} |\lambda |^2 \left(2 |\lambda|^4+|\nu|
   ^2\right)-|\lambda|^2 |\nu|^6 \left(|\lambda|^4+2|\nu|^2\right)\right)}{\left(|\lambda|^4+|\nu|^2\right) \left(e^{2 a_0
   /3}+|\lambda|^2 |\nu|^2\right) \left(e^{2 a_0 /3} |\lambda|^2+|\nu|^4\right)},\\
&([\varphi_i(g\cdot p_j)](1,1,1)^{T})_2 = \frac{3 \left(|\lambda|^4-|\nu|^2\right) \left(2 e^{2 a_0 /3} |\nu|^2 \left(|\lambda|^4+|\nu| ^2\right)+e^{4 a_0 /3} |\lambda|
   ^2+|\lambda|^2 |\nu|^6\right)}{\left(|\lambda|^4+|\nu|^2\right) \left(e^{2 a_0 /3}+|\lambda|^2 |\nu|^2\right) \left(e^{2
   a_0 /3} |\lambda|^2+|\nu|^4\right)}.    
\end{align*}
Thus, setting $\nu = \lambda^2$, that is restricting to the $1$-parameter subgroup 
\begin{equation*}
 \lambda \cdot[z_0:z_1:z_2] = [\lambda z_0 : \lambda z_1: \lambda^{-2} z_2],
\end{equation*}
we obtain 
\begin{equation*}
[\varphi_i(p_j)](1,1,1)^{T} = \left(\frac{9 \left(e^{2 a_0 /3}-|\lambda|^6\right)}{2 \left(e^{2 a_0 /3}+|\lambda| ^6\right)},0\right),
\end{equation*}
from which we see that the balancing condition can always be achieved by choosing
\begin{equation*}
|\lambda| = e^{a_0 / 9},\,|\nu| = e^{2 a_0 / 9}.
\end{equation*}
\begin{rmk} Note in particular that for $H = 0$, i.e. for the polarisation $-K_X$, the base locus $\Bs(\fd_{\Delta})$ is already balanced. 
\end{rmk}
We have thus established
\begin{lemma} Let $p\!: Z'\to T_{\nabla}$ be the blowup of the base locus $\Bs(\fd_{\Delta})$. Then, each K\"ahler class  
\begin{equation*}
p^*[\omega_{\nabla}] -\eps^2 \sum_{i = 1} [E_i],\, \eps \in (0, \bar{\eps})\text{ for some } \bar{\eps} > 0,
\end{equation*}
where $E_i \subset Z'$ denote the exceptional divisors of $p$, admits a (unique) orbifold cscK metric $\omega_{Z', \eps}$. In particular, the pair $(Z', [\omega_{Z', \eps}])$ is $K$-stable. Moreover, $\bar{\eps} > 0$ can be chosen uniformly when $-K_X + H$ lies in a bounded neighbourhood of $-K_X$ in the K\"ahler cone.
\end{lemma}
\subsubsection{Kempf-Ness Theorem and Hilbert-Mumford Criterion}
By the Kempf-Ness Theorem, the vanishing moment map condition is equivalent to the polystability of the $G$-orbit of the base locus $\Bs(\fd_{\Delta})$, regarded as a cycle, definining a point in the relevant Chow variety (see \cite{Mumford_GIT}, Theorem 8.3 for a statement of the Kempf-Ness Theorem for arbitrary K\"ahler forms representing a rational polarisation). By the Hilbert-Mumford Criterion, the latter can be checked using one-parameter subgroups (1PS) of $G$. The Hilbert-Mumford weight for the cycle is simply the sum of punctual Hilbert-Mumford weights (for more details see e.g. \cite{Stoppa_JAG}).

In our case, a 1PS of $G \cong (\C^*)^2$ is obtained by specialising $\lambda = t^a$, $\mu = t^b$, acting by
\begin{align*}
&t\cdot\Bs(\fd_{\Delta}) = t\cdot\{p_1,\,p_2,\,p_3\} \\
&= \{[0:0:t^{3(-a+b)}:- t^{-3b}e^{ a_0}],\,[0:t^{3a}:0:-t^{-3b}e^{ a_0}],\,[0:t^{3a}:-t^{3(-a+b)}:0]\} 
\end{align*} 
The computation of limit cycles as $t \to 0$ and Chow weights in the (nontrivial) chambers can be summarised as follows (we use the convention that stable points have negative Chow weights):
{\tiny
\begin{align*}
\begin{matrix*}[l]
&a > -b,\, -a + b > -b,\,a>-a+b: &a > -b,\, -a + b > -b,\,a < -a+b:\\
&p'_1 = [0:0:0:1],\,p'_2=[0:0:0:1],\,p'_3 = [0: 0: 1:0]&p'_1 = [0:0:0:1],\,p'_2=[0:0:0:1],\,p'_3 = [0: 1:0:0]\\
&\Ch = -3(a+b) < 0.&\Ch = 3(a - 2b) < 0.
\end{matrix*}
\end{align*}
\begin{align*}
\begin{matrix*}[l]
&a > -b,\, -a + b < -b,\,a > -a+b:&a < -b,\, -a + b > -b,\,a < -a+b:  \\
&p'_1 = [0:0:1:0],\,p'_2=[0:0:0:1],\,p'_3 = [0: 0:1:0]&p'_1 = [0:0:0:1],\,p'_2=[0:1:0:0],\,p'_3 = [0: 1:0:0]\\
&\Ch = 3(-2a + b) < 0.&\Ch =   3(2a-b) < 0.
\end{matrix*}
\end{align*}
\begin{align*}
\begin{matrix*}[l]
&a < -b,\, -a + b < -b,\,a > -a+b:&a < -b,\, -a + b < -b,\,a < -a+b: \\
&p'_1 = [0:0:1:0],\,p'_2=[0:1:0:0],\,p'_3 = [0: 0:1:0]&p'_1 = [0:0:1:0],\,p'_2=[0:1:0:0],\,p'_3 = [0: 1:0:0]\\
&\Ch = 3(-a +2b) < 0.&\Ch = 3(a+b) < 0.
\end{matrix*}
\end{align*}
}
\subsection{Third step: crepant resolution}
The final step in the canonical construction consists in passing to a crepant resolution of $Z'$. In dimension $2$, this is given simply by blowing up the orbifold points twice. Each singular point is replaced by two $(-2)$ curves. The pencil $\fd_{\Delta}$ induces a regular morphism $f\!: Z \to \PP^1$, which is a rational elliptic surface and so a log Calabi-Yau compactified toric LG model for $\PP^2$ (see e.g. \cite{Przyjalkowski_CYLG}, Section 3 or \cite{Petracci_reflexive}, Proposition 4.8). 

We would like to also perform this step \emph{metrically}, i.e. to endow the crepant resolution $\pi\!: Z \to Z'$ with \emph{natural, canonical metrics $\omega_{Z, \eps, \delta}$ obtained from the orbifold KE metric $\omega_{Z', \eps}$ on $Z'$}.

Since $Z'$ has discrete automorphisms, we may apply the general result \cite{ArezzoPacard}, Theorem 1.3. This shows that the crepant resolution $Z$ admits a unique cscK metric in each K\"ahler class
\begin{equation*}
[\omega_{Z, \eps, \delta}] = \pi^*[\omega_{Z', \eps}] + \delta^2 \sum^3_{i=1}  [\tilde{\eta}_i],\, \delta \in (0, \bar{\delta})\text{ for some } \bar{\delta} > 0,
\end{equation*}  
where $[\tilde{\eta}_i]$, for $i=1, 2, 3$, is the image in $Z$ of the K\"ahler class $[\eta]$ of an \emph{ALE Ricci flat resolution of $\C^2/(\Z/3)$ with vanishing ADM mass}. 
\begin{cor} Let $\pi\!: Z \to Z'$ be the crepant resolution. Then, each K\"ahler class  
\begin{equation*}
[\omega_{Z, \eps, \delta}] = \pi^*[\omega_{Z', \eps}] + \delta^2 \sum^3_{i=1}  [\tilde{\eta}_i],\, \delta \in (0, \bar{\delta})\emph{ for some } \bar{\delta} > 0,
\end{equation*}
admits a (unique) metric $\omega_{Z, \eps, \delta}$. In particular, the pair $(Z, [\omega_{Z, \eps, \delta}])$ is $K$-stable. Moreover, the parameters $\eps, \delta$ can be chosen uniformly when $-K_X + H$ lies in a bounded neighbourhood of $-K_X$ in the complexified K\"ahler cone. 

Thus, each admissible choice of parameters $\eps, \delta$ gives a well-defined holomorphic map from such a bounded neighbourhood to a separated moduli space $\mathfrak{M}$ containing cscK (hence $K$-stable) compactified Landau-Ginzburg models for $X$. 
\end{cor}
\section{$X = T_{P_{6a}} = S_6$}\label{Deg6Sec}
The degree $6$ del Pezzo $X = S_6$ is the unique smooth toric K\"ahler-Einstein del Pezzo surface. It is isomorphic to the blowup of $\PP^2$ in three points in general position. We denote the exceptional divisors by $E_i$. We can choose a toric presentation such that the corresponding reflexive (smooth) polygon is given by $\Delta = P_{6a}$ in \cite{Petracci_reflexive}, Fig. 2. This in also the unique smooth del Pezzo polygon such that the polar dual $\nabla = (P_{6a})^{\circ}$ is equivalent to $P_{6a}$ up to a $GL(2, \Z)$ transformation, so we have $T_{\nabla} \cong S_6$. 

According to \cite{Przyjalkowski_CYLG}, Section 2, Facts 1.-2., $T_{\nabla}$ is anticanonically embedded in a projective space $\PP^6$ with homogeneous coordinates $x_i$, $i = 0,\cdots, 6$ corresponding to the integral points $c_i$ of $\Delta$, and with homogeneous ideal cut out by the equations
\begin{equation*}
\prod_i x^{\alpha_i}_i = \prod_i x^{\beta_i}_i  
\end{equation*}    
induced by the monoid relations
\begin{equation*}
\sum_i  \alpha_i c_i = \sum_j \beta_j  c_j,\,\alpha_i,\beta_j > 0.  
\end{equation*}    
Setting $c_0 = 0$, one checks that these relations are generated by
\begin{equation*}
c_0 + c_i = c_{i-1} + c_{i+1},\,i=1,\cdots,6\text{ (cyclic)};\, c_i + c_{i + 3} = 2 c_0,\,i=1,\cdots, 3,
\end{equation*}
and so $T_{\nabla} \subset \PP^6$ is cut out by the corresponding quadratic equations
\begin{equation}\label{quadEqsP6}
x_{i-1} x_{i+1} = x_0 x_i,\,i=1,\cdots,6 \text{ (cyclic)};\,x_i x_{i+3} = x^2_0,\,i=1,\cdots, 3.
\end{equation}
By \cite{Przyjalkowski_CYLG}, Section 2, Fact 4, or \cite{Petracci_reflexive}, Construction 4.1, the toric boundary of $T_{\nabla}$ is cut out by the further equation $x_0 = 0$.

Choose an effective divisor
\begin{equation*}
H = a_0 l + a_1 E_1 + a_2 E_2 + a_3 E_3.
\end{equation*}
A toric LG potential $f_{S_6, H}$ for $T_{\Delta} = S_6$ is given by Givental's construction \cite{Givental_toric}. We follow the inductive presentation given in \cite{Przyjalkowski_CYLG}, Section 3 (this is convenient for later purposes). So, we start from the potential for $S_7$, the blowup of $\PP^2$ in two distinct points, with divisor
\begin{equation*}
H' = a_0 l + a_1 E_1 + a_2 E_2.
\end{equation*} 
According to \cite{Przyjalkowski_CYLG}, Example 24, this can be presented as
\begin{equation*}
f_{S_7, H'} = x + y + e^{-a_0} \frac{1}{x y} + e^{-a_0-a_1} \frac{1}{ y} + e^{-a_2} x y.
\end{equation*} 
Applying the general inductive relation in \cite{Przyjalkowski_CYLG}, Section 3, Construction 16 we thus obtain
\begin{align}\label{fS6}
\nonumber f_{S_6, H} &= f_{S_7, H'} + e^{-a_0} e^{-a_3}\frac{1}{x}\\
&=x + y + e^{-a_0} \frac{1}{x y} + e^{-a_0 -a_3}\frac{1}{x} + e^{-a_0-a_1} \frac{1}{ y} + e^{-a_2} x y.
\end{align}  
We see that the Laurent polynomial $f_{S_6, H}$ is the restriction to the structure torus of $T_{\nabla}$ of the birational map $T_{\nabla} \dashrightarrow \PP^1$ defined by the pencil
\begin{equation*}
\fd_{\Delta} := |x_1 + x_3 + e^{-a_0} x_5 + e^{-a_0 -a_3} x_4 + e^{-a_0-a_1} x_6 + e^{-a_2} x_2, x_0|,
\end{equation*}    
while the anticanonical embedding of the structure torus is given by
\begin{equation*}
(x, y) \mapsto [1: x : xy: y:  \frac{1}{x}: \frac{1}{xy}:\frac{1}{y}].
\end{equation*}

Using the quadratic equations, we check that the toric boundary components are smooth rational curves presented as 
\begin{equation*}
C_i = \{x_j = 0, j \neq i, i+1\},\,i = 1,\ldots, 6;\,C_i \cong \PP[x_i:x_{i+1}] 
\end{equation*}
(alternatively this follows from \cite{Przyjalkowski_CYLG}, Section 2, Fact 3). Thus, each boundary component contains precisely one base point $p_i \in \Bs(\fd_{\Delta})$, and we compute explicitly
\begin{align*}
&p_1 = [0:1: -e^{a_2}:0:0:0:0],\,p_2 = [0: 0: 1: -e^{-a_2} : 0 : 0 : 0],\\ 
&p_3 = [0: 0 : 0: 1 : -e^{a_0+a_3} : 0 :0],\,p_4 = [0: 0: 0: 0: 1:-e^{-a_3}:0],\\
&p_5 = [0: 0: 0: 0 : 0 : 1 : -e^{a_1}],\,p_6 = [0: 1 : 0: 0 : 0 : 0 : -e^{a_1}].
\end{align*}
\subsection{Kempf-Ness Theorem}\label{Deg6KempfNessSec}
A maximal compact subgroup $K \subset \Aut_0(T_{\nabla})$, containing the compact structure torus, is precisely the structure torus. Its complexification $K^{\C} \cong (\C^*)^2$ acts on the dense open orbit as
\begin{equation*}
(\lambda, \nu) \cdot (x, y) = (\lambda x, \nu y) \mapsto [1: \lambda x :  \lambda \nu xy: \nu y:  \lambda^{-1}\frac{1}{x}: \lambda^{-1}\nu^{-1}\frac{1}{xy}:\nu^{-1}\frac{1}{y}]. 
\end{equation*}
This action is induced by the embedding $(\C^*)^2 \hookrightarrow SL(7, \C)$ given by
\begin{equation*}
\operatorname{diag}(1, \lambda,  \lambda \nu , \nu , \lambda^{-1}, \lambda^{-1}\nu^{-1}, \nu^{-1} ). 
\end{equation*}
A general $1$PS $\alpha \hookrightarrow (\C^*)^2$ is obtained by setting $\lambda = t^{a}$, $\nu = t^b$, acting in the representation above as
\begin{equation*}
\alpha = \operatorname{diag}(1, t^a,  t^{a+b} , t^b , t^{-a}, t^{-a-b}, t^{-b} ). 
\end{equation*}
Thus, the induced action on the base locus is given by
\begin{align*}
&t\cdot p_1 = [0:t^a: -e^{a_2}t^{a+b}:0:0:0:0],\,t\cdot p_2 = [0: 0: t^{a+b}: -e^{-a_2}t^{b} : 0 : 0 : 0],\\ 
&t\cdot p_3 = [0: 0 : 0: t^{b} : -e^{a_0+a_3} t^{-a}: 0 :0],\,t\cdot p_4 = [0: 0: 0: 0: t^{-a}:-e^{-a_3}t^{-a-b}:0],\\
&t\cdot p_5 = [0: 0: 0: 0 : 0 : t^{-a-b} : -e^{a_1}t^{-b}],\,t\cdot p_6 = [0: t^a : 0: 0 : 0 : 0 : -e^{a_1}t^{-b}],
\end{align*}
or equivalently
\begin{align*}
&t\cdot p_1 = [0: 1: -e^{a_2}t^{b}:0:0:0:0],\,t\cdot p_2 = [0: 0: t^{a}: -e^{-a_2} : 0 : 0 : 0],\\ 
&t\cdot p_3 = [0: 0 : 0: t^{a+b} : -e^{a_0+a_3}: 0 :0],\,t\cdot p_4 = [0: 0: 0: 0: t^b:-e^{-a_3} :0],\\
&t\cdot p_5 = [0: 0: 0: 0 : 0 : 1 : -e^{a_1} t^a],\,t\cdot p_6 = [0: t^{a+b} : 0: 0 : 0 : 0 : -e^{a_1}].
\end{align*}
\begin{lemma}\label{Deg6KempfNess} The $K^{\C}$-orbit of $\Bs(\fd_{\Delta})$ in the Chow variety is stable.
\end{lemma}
\begin{proof}
We apply the Hilbert-Mumford Criterion in the Chow variety. The computation of limit cycles as $t \to 0$ and Chow weights is summarised in the following tables.  
{\tiny
\begin{align*}
\begin{matrix*}[l]
&a, b > 0: & a, b < 0:\\
&p'_1 = [0: 1: 0:0:0:0:0],\, p'_2 = [0: 0: 0: 1 : 0 : 0 : 0], &p'_1 = [0: 0: 1:0:0:0:0],\, p'_2 = [0: 0: 1: 0 : 0 : 0 : 0],\\ 
&p'_3 = [0: 0 : 0: 0 : 1: 0 :0],\, p'_4 = [0: 0: 0: 0: 0:1 :0], &p'_3 = [0: 0 : 0: 1 : 0: 0 :0],\, p'_4 = [0: 0: 0: 0: 1:0 :0],\\
&p'_5 = [0: 0: 0: 0 : 0 : 1 : 0],\, p'_6 = [0: 0: 0: 0 : 0 : 0 : 1],&p'_5 = [0: 0: 0: 0 : 0 : 0: 1],\, p'_6 = [0: 1 : 0: 0 : 0 : 0 : 0],\\
&\Ch = -2(a+b) < 0.&\Ch = 2b < 0.
\end{matrix*}
\end{align*}
\begin{align*}
\begin{matrix*}[l]
& a > 0, b < 0, a+b > 0: & a > 0, b < 0, a+b < 0:\\
& p'_1 = [0: 0: 1:0:0:0:0],\, p'_2 = [0: 0: 0: 1 : 0 : 0 : 0],& p'_1 = [0: 0: 1:0:0:0:0],\, p'_2 = [0: 0: 0: 1 : 0 : 0 : 0],\\ 
& p'_3 = [0: 0 : 0: 0: 1: 0 :0],\, p'_4 = [0: 0: 0: 0: 1:0 :0], & p'_3 = [0: 0 : 0: 1 : 0: 0 :0],\, p'_4 = [0: 0: 0: 0: 1:0: 0],\\
& p'_5 = [0: 0: 0: 0 : 0 : 1 : 0],\,  p'_6 = [0: 0 : 0: 0 : 0 : 0 : 1],& p'_5 = [0: 0: 0: 0 : 0 : 1 : 0],\,  p'_6 = [0: 1 : 0: 0 : 0 : 0 : 0],\\
&\Ch = -2a < 0.&\Ch = 2b<0.
\end{matrix*}
\end{align*}
\begin{align*}
\begin{matrix*}[l]
&a < 0, b > 0, a+b>0: &a < 0, b > 0, a+b<0:\\
& p'_1 = [0: 1: 0:0:0:0:0],\, p'_2 = [0: 0: 1: 0 : 0 : 0 : 0],& p'_1 = [0: 1: 0:0:0:0:0],\, p'_2 = [0: 0: 1: 0 : 0 : 0 : 0],\\ 
& p'_3 = [0: 0 : 0: 0 : 1: 0 :0],\, p'_4 = [0: 0: 0: 0: 0:1 :0],& p'_3 = [0: 0 : 0: 1 : 0: 0 :0],\, p'_4 = [0: 0: 0: 0: 0:1 :0],\\
& p'_5 = [0: 0: 0: 0 : 0 : 0 : 1],\, p'_6 = [0: 0 : 0: 0 : 0 : 0 : 1],& p'_5 = [0: 0: 0: 0 : 0 : 0 : 1],\, p'_6 = [0: 1 : 0: 0 : 0 : 0 : 0],\\
&\Ch = -b<0.&\Ch = 2a < 0.
\end{matrix*}
\end{align*}
\begin{align*}
\begin{matrix*}[l]
&a = 0, b > 0:& a = 0, b < 0:\\
& p'_1 = [0: 1: 0:0:0:0:0],\, p'_2 = [0: 0: 1: -e^{-a_2} : 0 : 0 : 0],& p'_1 = [0: 0: 1:0:0:0:0],\, p'_2 = [0: 0: 1: -e^{-a_2} : 0 : 0 : 0],\\ 
& p'_3 = [0: 0 : 0: 0: 1: 0 :0],\, p'_4 = [0: 0: 0: 0: 0:1 :0],& p'_3 = [0: 0 : 0: 1 : 0 : 0 :0],\, p'_4 = [0: 0: 0: 0: 1:0 :0],\\
& p'_5 = [0: 0: 0: 0 : 0 : 1 : -e^{a_1} ],\, p'_6 = [0: 0 : 0: 0 : 0 : 0 : 1],& p'_5 = [0: 0: 0: 0 : 0 : 1 : -e^{a_1}  ],\, p'_6 = [0: 1 : 0: 0 : 0 : 0 : 0],\\
&\Ch = -2b < 0.&\Ch = 2b < 0. 
\end{matrix*}
\end{align*}
\begin{align*}
\begin{matrix*}[l]
& a > 0, b = 0:& a < 0, b = 0:\\
& p'_1 = [0: 1: -e^{a_2}:0:0:0:0], p'_2 = [0: 0: 0: 1 : 0 : 0 : 0], & p'_1 = [0: 1: -e^{a_2}:0:0:0:0],\, p'_2 = [0: 0: 1: 0 : 0 : 0 : 0],\\ 
& p'_3 = [0: 0 : 0: 0: 1: 0 :0],\, p'_4 = [0: 0: 0: 0: 1:-e^{-a_3} :0], & p'_3 = [0: 0 : 0: 1 : 0: 0 :0],\, p'_4 = [0: 0: 0: 0: 1:-e^{-a_3} :0],\\
& p'_5 = [0: 0: 0: 0 : 0 : 1 : 0],\, p'_6 = [0: 0 : 0: 0 : 0 : 0 : 1],& p'_5 = [0: 0: 0: 0 : 0 : 0 : 1],\, p'_6 = [0: 1 : 0: 0 : 0 : 0 : 0],\\
&\Ch = -2a<0.&\Ch = 2a<0.
\end{matrix*}
\end{align*}
}
\end{proof}
The claim of Theorem \ref{MainThmDelPezzo} for our $X$ now follows from the (generalised) Arezzo-Pacard Theorem, as discussed in detail in Section \ref{PP2Sec}. 

\section{$T_X = T_{P_9}$}\label{CubicSection}
This is the mirror situation with respect to Section \ref{PP2Sec}. Now $T_X$ is the singular, Gorenstein toric del Pezzo $T_{\Delta}$ with fan polytope $\Delta = P_9$ in the list \cite{Petracci_reflexive}, Fig. 2, i.e., as we know from Section \ref{PP2Sec}, a singular cubic surface with 3 orbifold points. 

Clearly $T_X$ is a Gorenstein toric degeneration for a smooth cubic surface $X = S_3$ and so gives rise to a toric Landau-Ginzburg model mirror to $S_3$. Its crepant resolution $\widetilde{T}_{\Delta}$ is a weak del Pezzo surface and plays an auxiliary role in the construction of the LG potential. We fix the complexified divisor $H = a_0 l + \sum^6_{i = 1}a_i E_i = H' +\sum^6_{i = 4}a_i E_i$. 

Up to a reflection, the polygon $P_9$ is obtained from $P_6$ by adding three lattice vertices, where, again up to a reflection, we take here $P_6$ to be the lattice polygon with ordered set of vertices
\begin{align*}
\{v_1 = (0,1),\,v_2 = (1,1),\,v_3 = (1,0),\,v_4 = (0,-1),\,v_5 = (-1,-1),\,v_6 = (-1,0)\}.
\end{align*}
Then $P_9$ is obtained from $P_6$ by adding the vertices
\begin{align*}
\{v_7 = (1,2),\,v_8 = (1,-1),\,v_9 = (-2,-1)\}.
\end{align*}

Thus, the potential $f_{\widetilde{T}_{\Delta}, H}$ can be obtained from $f_{S_6, H'}$ by applying the general inductive relation in \cite{Przyjalkowski_CYLG}, Section 3, Construction 16 independently at each new vertex, 
\begin{align}\label{fNodalCubic}
& f_{\widetilde{T}_{\Delta}, H} = f_{S_6, H'} + e^{-a_2} e^{-a_4} x y^2 + e^{-a_0-a_1} e^{-a_5} \frac{x}{y} + e^{-a_0} e^{-a_0 -a_3} e^{-a_6} \frac{1}{x^2 y}.  
\end{align}
However, since $P_9$ is singular, i.e. contains non-vertex lattice points on its boundary, this is only an auxiliary potential, and the actual potential $f_{S_3, H}$ is obtained by applying the modification discussed in \cite{Przyjalkowski_CYLG}, Section 3, Construction 16 independently at each non-vertex lattice point on the boundary. Thus, we consider the positively (clockwise) ordered sets of lattice points in each boundary component, 
\begin{align*}
(w_0, w_1, w_2, w_3) = (v_8, v_4, v_5, v_9),\,(v_7, v_2, v_3, v_8),\,(v_9, v_6, v_1, v_7).
\end{align*}
Then, setting
\begin{equation*}
f_{\widetilde{T}_{\Delta}, H} = \sum_i m_{v_i} x^{v_i},
\end{equation*}
the coefficient of $f_{S_3, H}$ at $w_i$, for $i = 1, 2$, is given by the coefficient of $s^i$ in the polynomial
\begin{equation}\label{GiventalModification}
m_{w_0}\left(1+\frac{m_{w_1}}{m_{w_0}} s\right)\left(1+\frac{m_{w_2}}{m_{w_1}}s\right)\left(1+\frac{m_{w_3}}{m_{w_2}}s\right).
\end{equation}
Using \eqref{fS6}, \eqref{fNodalCubic} and \eqref{GiventalModification} we compute
\begin{align*}
& f_{S_3, H} = e^{-a_2 - a_4} x y^2 + e^{-a_0-a_1-a_5} \frac{x}{y} +   e^{-2a_0 -a_3 -a_6} \frac{1}{x^2 y} + \sum^{3}_{i = 1} f^{(i)}_{S_3, H},
\end{align*}
where the non-vertex boundary contributions are given by
\begin{align*}
&f^{(1)}_{S_3, H} = e^{-2  a_0 - a_2 -  a_3 -  a_4 -  a_6 }
   \left(e^{ a_0 +  a_2 + a_3 + a_4 }+e^{ a_0 + a_2 +  a_4 + a_6 }+1\right)\frac{1}{x}\\
&+e^{- a_0 -  a_2 -  a_3 - a_4 -  a_6 }
   \left(e^{ a_0 + a_2  +  a_3 +  a_4 +  a_6 }+e^{ a_3 }+e^{ a_6}\right)y,\\ 
&f^{(2)}_{S_3, H} =  e^{- a_0 - a_1 -  a_2 -  a_4 -  a_5}
   \left(e^{ a_0 + a_1 +  a_2 +  a_5}+e^{ a_0 + a_1 +  a_4 +  a_5}+1\right)x y\\
& + e^{- a_0 -  a_1 - a_2 - a_4 -  a_5}
   \left(e^{ a_0 + a_1 + a_2 + a_4 + a_5}+e^{ a_2 }+e^{ a_4 }\right)x,\\
&f^{(3)}_{S_3, H} = e^{-2 a_0 -  a_1- a_3 - a_5 - a_6}
   \left(e^{a_0 + a_1 + a_3 +  a_6} + e^{ a_0 +  a_3 + a_5 + a_6}+1\right)\frac{1}{y}\\
&+ e^{-2 a_0 - a_1 - a_3 -  a_5 - a_6 }
   \left(e^{ a_0 + a_1 +  a_3 + a_5 + a_6}+e^{ a_1}+ e^{ a_5}\right)\frac{1}{x y}.
\end{align*}
The dual of the reflexive polygon $P_9$ is equivalent to $P_3$. Therefore, the anticanonical embedding $T_{\nabla} \hookrightarrow \PP^{9}$ coincides with the Veronese map of degree $3$. Its ideal is generated by quadrics, given by \eqref{quadEqsP6} (corresponding to our previous homogeneous monoid relations for $P_6$) together with the new monoid relations involving the added vertices $v_7, v_8, v_9$. They comprise those involving $x_0$,
\begin{align*}
x_1 x_2 = x_0 x_7,\, x_3 x_4 = x_0 x_8,\,x_5 x_6 = x_0 x_9,\,
\end{align*}
together with a set of twisted cubic relations in the set of variables corresponding to each edge of $P_9$, namely
\begin{align*}
& Q_1 = \{x_1 x_9 - x^2_6,\,x_6 x_7 - x^2_1,\,x_7 x_9 - x_1 x_6\},\\
& Q_2 = \{x_3 x_7 - x^2_2,\,x_2 x_8 - x^2_3,\,x_7 x_8 - x_2 x_3\},\\
& Q_3 = \{x_5 x_8 - x^2_4,\,x_4 x_9 - x^2_5,\,x_8 x_9 - x_4 x_5\}.
\end{align*}
The toric boundary is cut out by $x_0 = 0$ and so it is given by the union of the twisted cubics in projective subspaces of $\PP^9$,   
\begin{align*}
& V(Q_1) = \{[s^2 t: s t^2: s^3:t^3]\}\subset \PP[x_1 : x_6: x_7: x_9],\\
& V(Q_2) = \{[s^2 t: s t^2: s^3:t^3]\}\subset \PP[x_2:x_3:x_7:x_8],\\
& V(Q_3) = \{[s^2 t: s t^2: s^3:t^3]\}\subset \PP[x_4:x_5:x_8:x_9].
\end{align*}
So the restriction of $f_{S_3, H}$ to the torus orbit in each boundary component is given by
\begin{align*}
& f_{S_3, H}|_{V(Q_1)} =  e^{-2a_0 -a_3 -a_6}\\
&+e^{-2  a_0 - a_2 -  a_3 -  a_4 -  a_6 }
   \left(e^{ a_0 +  a_2 + a_3 + a_4 }+e^{ a_0 + a_2 +  a_4 + a_6 }+1\right)s\\
&+e^{- a_0 -  a_2 -  a_3 - a_4 -  a_6 }
   \left(e^{ a_0 + a_2  +  a_3 +  a_4 +  a_6 }+e^{ a_3 }+e^{ a_6}\right)s^2\\
&+e^{-a_2 - a_4} s^3,\\
& f_{S_3, H}|_{V(Q_2)} = e^{-a_0 -a_1 -a_5}\\
& + e^{- a_0 -  a_1 - a_2 - a_4 -  a_5}
   \left(e^{ a_0 + a_1 + a_2 + a_4 + a_5}+e^{ a_2 }+e^{ a_4 }\right)s\\
&+e^{- a_0 - a_1 -  a_2 -  a_4 -  a_5}
   \left(e^{ a_0 + a_1 +  a_2 +  a_5}+e^{ a_0 + a_1 +  a_4 +  a_5}+1\right)s^2\\
 &+ e^{-a_2 - a_4} s^3,\\
 & f_{S_3, H}|_{V(Q_3)} = e^{-2a_0 -a_3 -a_6}\\
 &+ e^{-2 a_0 - a_1 - a_3 -  a_5 - a_6 }
   \left(e^{ a_0 + a_1 +  a_3 + a_5 + a_6}+e^{ a_1}+ e^{ a_5}\right)s\\
 &+ e^{-2 a_0 -  a_1- a_3 - a_5 - a_6}
   \left(e^{a_0 + a_1 + a_3 +  a_6} + e^{ a_0 +  a_3 + a_5 + a_6}+1\right)s^2\\
&+e^{-a_0-a_1-a_5} s^3.
\end{align*}
Using \eqref{GiventalModification} we compute, up to a factor in $\C^*$,  
\begin{align*}
& f_{S_3, H}|_{V(Q_1)} \propto \left( e^{ a_0 + a_3 }s+1\right) \left( e^{ a_0 + a_6 }s+1\right)
   \left(e^{ a_2 +  a_4}+s\right),\\
& f_{S_3, H}|_{V(Q_2)} \propto \left(e^{ a_2 }+s\right) \left(e^{ a_4 }+s\right) \left(
   e^{ a_0 + a_1 + a_5}s+1\right),\\
& f_{S_3, H}|_{V(Q_3)} \propto \left(e^{ a_1 }+s\right) \left(e^{ a_5 }+s\right) \left(
   e^{ a_0 +  a_3 + a_6}s+1\right).      
\end{align*}
Therefore, the base locus of the pencil defined by $f_{S_3, H}$ is given by the cycle
\begin{align*}
\Bs(\fd_{\Delta}) = \{p_1, \ldots, p_9\} \subset T_{\nabla} \cong \PP^2 \subset \PP^9
\end{align*}
where
\begin{align}\label{CubicArrangement}
\nonumber&p_1 = [e^{-2a_0 - 2a_3}: -e^{-a_0 - a_3}: -e^{-3a_0 - 3a_3}:1] \in \PP[x_1 : x_6: x_7: x_9],\\
\nonumber&p_2 = [e^{-2a_0 - 2a_6}: -e^{-a_0 - a_6}: -e^{-3a_0 - 3a_6}:1] \in \PP[x_1 : x_6: x_7: x_9],\\ 
\nonumber&p_3 = [e^{-2a_2 - 2a_4}: -e^{-a_2 - a_4}: -e^{-3a_2 - 3a_4}:1] \in \PP[x_1 : x_6: x_7: x_9],\\
\nonumber&p_4 = [e^{-2 a_2}: -e^{-a_2 }: -e^{-3 a_2 }:1] \in  \PP[x_2:x_3:x_7:x_8],\\
\nonumber&p_5 = [e^{-2 a_4}: -e^{-a_4 }: -e^{-3 a_4}:1] \in  \PP[x_2:x_3:x_7:x_8],\\ 
\nonumber&p_6 = [e^{-2 (a_0+a_1+a_5) }: -e^{-(a_0+a_1+a_5) }: -e^{-3 (a_0+a_1+a_5)}:1] \in  \PP[x_2:x_3:x_7:x_8],\\
\nonumber&p_7 = [e^{-2 a_1}: -e^{- a_1}: -e^{-3 a_1}:1] \in  \PP[x_4:x_5:x_8:x_9],\\
\nonumber&p_8 = [e^{-2 a_5}: -e^{- a_5}: -e^{-3 a_5}:1] \in  \PP[x_4:x_5:x_8:x_9],\\ 
&p_9 = [e^{-2 (a_0 +  a_3 + a_6)}: -e^{- (a_0 +  a_3 + a_6)}: -e^{-3 (a_0 +  a_3 + a_6)}:1] \in \PP[x_4:x_5:x_8:x_9].
\end{align}
It follows that the locus 
\begin{equation*}
\mathcal{D} := \{p_i = p_j \text{ for some } i \neq j \} \subset \Ka(S_3)_{\C}
\end{equation*}
is given by the intersection with $\Ka(S_3)_{\C}$ of the \emph{explicit, affine hyperplane arrangement} in $H^{1,1}(S_3, \C)$ determined by \eqref{CubicArrangement}. Fix an arbitrary norm on $H^{1,1}(S_3, \C)$, and consider the open subset  
\begin{equation*}
\mathcal{D}_{> \delta} := \{ \operatorname{dist}([\omega_{\C}], \mathcal{D}) > \delta > 0\} \subset \Ka(S_3)_{\C}.  
\end{equation*}
For fixed $\delta > 0$, the points $\{p_i\}$ is the support of the base locus corresponding to $[\omega_{\C}]\in\mathcal{D}_{>\delta}$ are uniformly separated. A straightforward application of the Hilbert-Mumford criterion to configurations of points in $T_{\nabla} = \PP^2$, polarised by $\mathcal{O}_{\PP^2}(1)$, shows that $\{p_i\}$ is a stable point in the relevant Chow variety. The Kempf-Ness theorem gives a zero of the moment map in its orbit. The Arezzo-Pacard Theorem then holds for a uniform $\varepsilon > 0$, depending only on $\delta$.  
\begin{cor}\label{MirrorCubicGeneric} Let $Y = \Bl_{\Bs(\fd_{\Delta})} T_{\nabla}$ denote the above compactified toric Landau-Ginzburg model of the cubic surface $S_3$. There exists $\varepsilon > 0$, depending only on $\delta > 0$, such that the K\"ahler class 
\begin{equation*}
[\omega_{\epsilon}] = \pi^* c_1(\mathcal{O}_{\PP^2}(1)) - \varepsilon^2 \sum^9_{i = 1} c_1(\mathcal{O}_{Y}(E_i))
\end{equation*}
contains a unique cscK representative for all complex structures on $Y$ corresponding to $[\omega_{\C}] \in \mathcal{D}_{> \delta}$.
\end{cor}
\subsection{Anticanonical class}\label{C1StrataSec} However, this is not sufficient for our purposes since the anticanonical class $c_{1}(S_3)$ lies in the singular locus $\mathcal{D}$. Indeed, $c_1(S_3)$ corresponds to the choice $a_i = 0,\,i = 0,\ldots, 6$ and so to the cycle
\begin{equation*}
\Bs( c_1(S_3)) := 3 q_1 + 3 q_2 + 3 q_3 := 3 p_1 + 3p_4 + 3p_7.
\end{equation*}
The corresponding complex structure on $Y$ is given by an iterated blowup of $T_{\nabla} \cong \PP^2$, first along the reduced cycle $\Bs(c_1(S_3))^{\operatorname{red}}$, then at the proper transforms of the toric boundary of $T_{\nabla}$, with exceptional divisors $E_{i,j}$ lying over $q_i$, such that
\begin{equation*}
E^2_{i, 1} = E^2_{i, 2} = -2,\,E^2_{i, 3} = -1,\,E_{i, j} . E_{i,  j+1 } = 1.
\end{equation*} 
We can construct stable polarisations on $Y$ as follows. 

Firstly, 
\begin{equation*}
Y_1 := \Bl_{\Bs(c_1(S_3))^{\operatorname{red}}} T_{\nabla} = \Bl_{q_1 + q_2 + q_3} T_{\nabla} 
\end{equation*}
is isomorphic to the degree $6$ del Pezzo $S_6$ discussed in Section \ref{Deg6Sec}. Thus, $Y_1$ is isomorphic to the blowup of $\PP^2$ in $3$ torus-fixed points. According to \cite{ArezzoPacardSinger, Gabor_blowup}, the class 
\begin{equation*}
\pi^*c_1(\mathcal{O}_{\PP^2}(1)) - \varepsilon^2_1 \sum_i [E_i]
\end{equation*} 
on $Y_1$ (with the obvious notation for exceptional divisors) admits an extremal metric. Since the Futaki character vanishes in this class (see e.g. \cite{SektnanTipler_3folds}, Section 4.2), this extremal metric is in fact cscK.

The second blowup $Y_2$ in our construction is isomorphic to blowing up $S_6$ at suitable points $q'_1,\,q'_2,\,q'_3$ contained in its toric boundary components $C_1, C_3, C_5$ described in Section \ref{Deg6Sec}. Recall we have
\begin{equation*}
C_1 = \PP[x_1:x_2],\,C_3 = \PP[x_3:x_4],\,C_5 = \PP[x_5:x_6]
\end{equation*}
with respect to the anticanonical embedding $S_6 \subset \PP[x_0:\cdots: x_6]$. By construction, the points $q'_1,\,q'_2,\,q'_3$ are \emph{not} fixed for the given toric structure on $S_6$. Thus, we write
\begin{align*}
q'_1 = [0:1: \beta_1: 0:0:0:0],\,q'_2 = [0: 0 : 0: 1 : \beta_2 : 0 :0],\,q'_3 = [0: 0: 0: 0 : 0 : 1 : \beta_3],
\end{align*}
for suitable $\beta_i$. As in Section \ref{Deg6KempfNessSec}, the action of a 1PS is given by
\begin{align*}
&t\cdot q'_1 = [0: 1: t^{b} \beta_1:0:0:0:0], \\ 
&t\cdot q'_2 = [0: 0 : 0: t^{a+b} : \beta_2: 0 :0], \\
&t\cdot q'_3 = [0: 0: 0: 0 : 0 : 1 :  t^a \beta_3].
\end{align*}
\begin{lemma} The $K^{\C}$-orbit of $\sum_i q'_i$ in the Chow variety is stable with respect to the linearisation induced by $\pi^*c_1(\mathcal{O}_{\PP^2}(3)) - \sum_i [E_i]$.
\end{lemma}
\begin{proof} This follows from precisely the same computation as Lemma \ref{Deg6KempfNess} (reading only the first columns in each case).
\end{proof}
\begin{cor} The $K^{\C}$-orbit of $\sum_i q'_i$ in the Chow variety is stable with respect to the linearisation induced by $\pi^*c_1(\mathcal{O}_{\PP^2}(1)) - \varepsilon^2_1\sum_i [E_i]$.
\end{cor}
\begin{proof} By the previous Lemma, the Chow weight of $\sum_i q'_i$ with respect to any 1PS and the linearisation induced by $\pi^*c_1(\mathcal{O}_{\PP^2}(3)) - \sum_i [E_i]$ is positive. So the same is true for the polarisation induced by $\varepsilon^2_1(\pi^*c_1(\mathcal{O}_{\PP^2}(3)) - \sum_i [E_i])$, and also for that induced by
\begin{equation*}
\pi^*c_1(\mathcal{O}_{\PP^2}(1)) - \varepsilon^2_1\sum_i [E_i] = (1 - 3\varepsilon^2_1)\pi^*c_1(\mathcal{O}_{\PP^2}(1)) + \varepsilon^2_1(\pi^*c_1(\mathcal{O}_{\PP^2}(3)) - \sum_i [E_i])
\end{equation*}
since $\pi^*c_1(\mathcal{O}_{\PP^2}(1))$ is pulled back from the base, so the Chow weight of the cycle $\sum_i q'_i$ with respect to $\pi^*c_1(\mathcal{O}_{\PP^2}(1))$ equals the Chow weight of its projection $\sum_i q_i$ with respect to $c_1(\mathcal{O}_{\PP^2}(1))$, and the latter Chow weight vanishes by polystability.   
\end{proof}
Write $E'_{i,j}$ for the exceptional divisors of the second blowup, with
\begin{equation*}
E'^2_{i, 1} = -2,\,E'^2_{i, 2} = -1,\,E'_{i, 1} . E'_{i,  2} = 1.
\end{equation*} 
By the above Corollary and the Arezzo-Pacard Theorem, the class on $Y_2$ given by
\begin{equation*}
\pi^*(\mathcal{O}_{\PP^2}(1)) - \sum_i[\varepsilon^2_1 E'_{i, 1} + (\varepsilon^2_1 + \varepsilon^2_2)E'_{i, 2}]  
\end{equation*}
admits a cscK representative for $0 < \varepsilon_2 \ll \varepsilon_1 \ll 1$. Now $Y$ is obtained by blowing up $Y_2$ at suitable points $q''_{i}$ of the exceptional divisors $E'_{i, 2}$. However, $Y_2$ has discrete automorphisms, so by the Arezzo-Pacard Theorem, the class on $Y$ given by
\begin{equation}\label{CubicClass}
\pi^*(\mathcal{O}_{\PP^2}(1)) - \sum_i [\varepsilon^2_1 E_{i, 1} + (\varepsilon^2_1 + \varepsilon^2_2) E_{i, 2} + (\varepsilon^2_1 + \varepsilon^2_2 +\varepsilon^2_3 )E_{i,3}]  
\end{equation} 
admits a unique cscK representative for $0 < \varepsilon_3 \ll \varepsilon_2 \ll \varepsilon_1 \ll 1$.
\subsection{Deformation theory}\label{DefoTheorySec}
A complex structure on $Y$ corresponding to $[\omega_{\C}] \in \mathcal{D}_{>\delta}$ can specialise to the complex structure corresponding to $c_1(S_3)$. Generically, up to a choice of indices, the map at the level of exceptional divisors is given by
\begin{align*}
& E_1 \mapsto \sum^3_{j = 1} E_{1, j},\quad E_2 \mapsto \sum^3_{j = 2} E_{1, j},\quad E_3 \mapsto E_{1, 3}, \\
& E_4 \mapsto \sum^3_{j = 1} E_{2, j},\quad E_5 \mapsto \sum^3_{j = 2} E_{2, j},\quad E_6 \mapsto E_{2, 3}, \\
& E_7 \mapsto \sum^3_{j = 1} E_{3, j},\quad E_8 \mapsto \sum^3_{j = 2} E_{3, j},\quad E_9 \mapsto E_{3, 3}.
\end{align*}
So, the class \eqref{CubicClass} is the specialisation of a class $[\omega_{\C}] \in \mathcal{D}_{>\delta}$ given by
\begin{equation*} 
\pi^*(\mathcal{O}_{\PP^2}(1)) -  \varepsilon^2_1 [E_1 + E_4 + E_7] - \varepsilon^2_2 [E_2 + E_5 + E_8] - \varepsilon^2_3 [E_3 + E_6 + E_9]. 
\end{equation*}
A slight variant of the proof of Corollary \ref{MirrorCubicGeneric} gives the following.
\begin{prop} For $0 < \varepsilon_3 \ll \varepsilon_2 \ll \varepsilon_1 \ll 1$, depending only on $\delta > 0$, the K\"ahler class 
\begin{equation}\label{NearbyClass} 
\pi^*(\mathcal{O}_{\PP^2}(1)) -  \varepsilon^2_1 [E_1 + E_4 + E_7] - \varepsilon^2_2 [E_2 + E_5 + E_8] - \varepsilon^2_3 [E_3 + E_6 + E_9] 
\end{equation}
contains a unique cscK representative for all complex structures on $Y$ corresponding to $[\omega_{\C}] \in \mathcal{D}_{> \delta}$.
\end{prop}
\begin{proof} We work in $\mathcal{D}_{> \delta}$ for a fixed $\delta > 0$. Firstly, as we already observed, the K\"ahler class on $\Bl_{p_1, p_4, p_7}\PP^2$ 
\begin{equation*}
\pi^*(\mathcal{O}_{\PP^2}(1)) -  \varepsilon^2_1 [E_1 + E_4 + E_7]  
\end{equation*}
is cscK for $0 < \varepsilon_1 \ll 1$. Let us denote by $K$ the compact torus fixing $p_1$, $p_4$, $p_7$ on $\PP^2$. A straightforward application of the Hilbert-Mumford criterion shows that the cycle $p_2 + p_5 + p_8$ on $\PP^2$ is polystable with respect to the full automorphism group of $(\PP^2, \mathcal{O}_{\PP^2}(1))$ and so by continuity its Chow weight as a cycle in 
\begin{equation*} 
(\Bl_{p_1, p_4, p_7}\PP^2, \pi^*(\mathcal{O}_{\PP^2}(1)) -  \varepsilon^2_1 [E_1 + E_4 + E_7])
\end{equation*} 
with respect to any nontrivial 1PS in $K$ is strictly positive for a uniform $0 < \varepsilon_1 \ll 1$. Then, by the Arezzo-Pacard Theorem, the K\"ahler class on $\Bl_{p_2, p_5, p_8} \Bl_{p_1, p_4, p_7}\PP^2$
\begin{equation*}
\pi^*(\mathcal{O}_{\PP^2}(1)) -  \varepsilon^2_1 [E_1 + E_4 + E_7] - \varepsilon^2_2 [E_2 + E_5 + E_8]  
\end{equation*}
is cscK for all $0 < \varepsilon_2 \ll \varepsilon_1 \ll 1$ depending only on $\delta > 0$. The claim now follows since $\Bl_{p_2, p_5, p_8} \Bl_{p_1, p_4, p_7}\PP^2$ has discrete automorphisms.
\end{proof}
Finally, we need to show that for some $\delta > 0$, sufficiently small, the open subset $\mathcal{D}_{> \delta}$ overlaps with an open neighbourhood of the class \eqref{CubicClass}, consisting of classes which admit a unique cscK representative. This can be achieved by deforming the complex structure of $Y$ given by the iterated blowup.

Applying the deformation theory of blowups in dimension two (see e.g. \cite{Hartshorne_defos}, Exercise 10.5), we obtain an isomorphism
\begin{align*}
H^1(Y, \mathcal{T}_Y) \cong H^1(Y_2, \mathcal{T}_{Y_2}) \bigoplus_i (\C^2)_{q''_i}
\end{align*}   
and an exact sequence
\begin{align*}
0 \to H^0(Y_1, \mathcal{T}_{Y_1}) \to \bigoplus_i (\C^2)_{q'_i} \to  H^1(Y_2, \mathcal{T}_{Y_2}) \to H^1(Y_1, \mathcal{T}_{Y_1}) \cong 0.
\end{align*}
Moreover, deformations are unobstructed, 
\begin{equation*}
H^2(Y, \mathcal{T}_Y) \cong H^2(Y_2, \mathcal{T}_{Y_2}) \cong H^2(Y_1, \mathcal{T}_{Y_2}) \cong 0 
\end{equation*}   
(this follows e.g. from \cite{Hartshorne_defos}, Exercise 10.5 (b)). So, small deformations of $Y$ correspond precisely to the blowup centers $\{p_i\}$, and the mirror map sends a neighbourhood of $c_1(S_3)$ in the complexified K\"ahler cone holomorphically to a neighbourhood of $Y$ in its smooth Kuranishi space. On the other hand, the automorphism group of $Y$ is discrete. Thus, by the results of \cite{FujikiSchumacher_moduli, LeBrunSimanca, Szekelyhidi_moduli}, there is an open neighbourhood of $Y$ in its smooth Kuranishi space admitting a unique cscK metric in the class \eqref{NearbyClass}. For $\delta > 0$ sufficiently small, the above neighbourhood of $c_1(S_3)$ intersects $\mathcal{D}_{> \delta}$ in an open subset as required.
\subsection{Extension to non-generic strata}\label{NonGenericSec}
The construction in the previous Section can be performed in order to glue to $\mathcal{D}_{> \delta}$ an open neighbourhood of any cscK complex structure on $Y$ corresponding to an iterated blowup. In order to extend  our map across strata of such complex structures on $Y$, we need to show that the parameter $\delta > 0$ can be chosen uniformly. However, this follows at once from the fact that the analytic estimates of \cite{Gabor_blowup} hold uniformly in $\Ka(S_3)_{\C}$ as long as the support of $\Bs(\fd_{\Delta})$ is bounded away from the torus-fixed points of the toric structure on $T_{\nabla}$.
\subsection{$K$-instability of adiabatic classes}\label{AdiabaticSec} 
This Section contains the proof of the Lemma stated in Remark \ref{AdiabaticRmk}.

Recall that the compactified mirror $f\!: Y \to \PP^1$, endowed with any complex structure corresponding to $[\omega_{\C}]$, is always a rational elliptic surface (in the sense of e.g. \cite{Petracci_reflexive}, Section 2.1). The (uniform) $K$-stability of rational elliptic surfaces is studied in \cite{Hattori_CYfibrations}, in the case when $Y$ is endowed with an \emph{adiabatic polarisation}, i.e. one of the form
\begin{equation*}
H + f^*c_1(\mathcal{O}_{\PP^1}(k))
\end{equation*}
for a fixed polarisation $H$ on $Y$ and with $k$ sufficiently large (depending on $(Y, H)$). These polarisations are very far from those considered in the present paper. Indeed, in the current case of the compactified toric Landau-Ginzburg $Y$ model mirror to a cubic surface $S_3$, it is even the case that the complex structure on $Y$ corresponding to $c_1(S_3) \in \Ka(S_3)_{\C}$ is \emph{$K$-unstable for all} $k \gg 1$. 

In order to see this, recall that according to \cite{Hattori_CYfibrations}, Corollary 1.5 $(i)$, if a rational elliptic surface has no multiple fibres and admits a fibre of type $II^*, III^*$ or $IV^*$ is the standard classification (see e.g. \cite{Petracci_reflexive}, Section 2.2), then it is $K$-unstable with respect to adiabatic classes. 

On the other hand, for our present case, the fibres of $f\!: Y \to \PP^1$ are studied in detail in \cite{Petracci_reflexive}, Section 5.16: it is shown there that there are precisely three singular fibres, of types $I_1, I_3$ and $IV^*$.  

Note that, for generic $[\omega_{\C}] \in \Ka(S_3)_{\C}$, the corresponding fibration $f\!: Y \to \PP^1$ has only fibres of type $I_N$ and so, by \cite{Hattori_CYfibrations}, Corollary 1.5 $(1)$, for any fixed polarisation $H$, the polarised surface $(Y, H + f^*c_1(\mathcal{O}_{\PP^1}(k)))$ is $K$-stable for $k \gg 1$. The problem is that, by the above instability result, it is not possible to choose $H$ along a family of complex structures approaching the mirror to $c_1(S_3)$ so that $k$ can be chosen uniformly. This problem does not arise for the polarisations we consider thanks to the uniform estimates in the Arezzo-Pacard and deformation theorems. 
\section{$X = T_{P_{4b}}$}
$X = T_{\Delta}$, $\Delta = P_{4b}$ is a smooth toric del Pezzo isomorphic to $\PP(\mathcal{O}(1)\oplus\mathcal{O})$ (or the blowup of $\PP^2$ at one point). Fix $H = a_0 l + a_1 E$. The corresponding toric LG potential is given by
\begin{equation*}
f = x + \frac{1}{y} + e^{-a_0}\frac{y}{x} + e^{-a_0-a_1} y.
\end{equation*}
The dual polytope is given by $\nabla = P_{8b}$. The singular, Gorenstein toric variety $T_{\nabla}$ is cut out in its anticanonical embedding $T_{\nabla} \subset \PP := \PP[x_0:x_1:x_2:x_3:x_4]$ by quadrics,
\begin{equation*}
\{ x_1 x_3 = x_0 x_4,\,x_2 x_4 = x^2_0 \}
\end{equation*}
where the monomials $x_i$, $i>0$ corresponds to the vertices of $P_{4b}$,
\begin{equation*}
\{v_1 = (1,0),\,v_2 = (0,-1),\,v_3 = (-1,1),\,v_4=(0,1)\}.
\end{equation*}
Thus, $T_{\nabla}$ is a Gorenstein toric degeneration of a smooth degree $4$ del Pezzo, and the Laurent polynomial $f$ extends to the anticanonical pencil on $T_{\nabla}$ given by
\begin{equation*}
\mathfrak{d}_{\Delta} := |x_1 + x_2 + e^{-a_0}x_3 + e^{-a_0-a_1} x_4 , x_0|
\end{equation*}
The divisor $(x_0)$ is the toric boundary of $T_{\nabla}$ in the fixed toric structure. Thus, the base locus $\Bs(\mathfrak{d}_{\Delta}) = \{p_1,\,p_2,\,p_3,\,p_4\}$ contains a reduced point for each boundary component and is given by 
\begin{align*}
&p_1 = [1:-1] \in \PP[x_1: x_2],\,p_2 = [1:-e^{-a_0}] \in \PP[x_2: x_3],\\
&p_3 = [1:-e^{-a_1}] \in \PP[x_3: x_4],\,p_4 = [1:-e^{-a_0-a_1}] \in \PP[x_1: x_4].
\end{align*}
The toric resolution $\widetilde{T}_{\nabla}$ has a fan corresponding to the maximal subdivision of $\nabla$, and so it is given by
\begin{equation*}
\widetilde{T}_{\nabla} = \Bl_{q'_2, q'_3} \Bl_{q_1, q_2, q_3}\PP^2 \cong \Bl_{q'_2, q'_3} S_6,
\end{equation*}
where $q'_2$, $q'_3$ are torus fixed points contained in the exceptional divisors over $q_2$, $q_3$. The compactified LG model is then given by
\begin{equation*}
Z = \Bl_{\Bs(\mathfrak{d}_{\Delta})} \widetilde{T}_{\nabla} \cong \Bl_{\Bs(\mathfrak{d}_{\Delta})}\Bl_{q'_2, q'_3} S_6.
\end{equation*}
By considering the subdivision of $\nabla$, we see that, following the notation of Section \ref{Deg6Sec}, we can choose the identification $\Bl_{q_1, q_2, q_3}\PP^2 \cong S_6$ so that the the basepoints $p_1$, $p_2$ are mapped to non torus-fixed points of the boundary components $C_1$, $C_2$, while the torus-fixed point $q'_2$ is mapped to $C_5 \cap C_6$. Let us consider the stability of a cycle on $S_6$ supported at $p_1, p_2, q'_2$, with respect to the (K\"ahler-Einstein) anticanonical polarisation. Since the Arezzo-Pacard Theorem allows nontrivial weights, it is possible to work with the cycle
\begin{equation*}
p_1 + p_2 + (1+ \delta)q'_2
\end{equation*}
where $\delta \in (0, 1)$ is a parameter (see \cite{Stoppa_JAG} for more details on such weighted cycles). The stability of this cycle can be studied following the computation in Lemma \ref{Deg6KempfNess}, however, the torus-fixed point $q'_2 = C_5 \cap C_6$ provides a fixed contribution $-\delta(a+b)$ to the Chow weight, so we have 
{\tiny
\begin{align*}
\begin{matrix*}[l]
&a, b > 0: & a, b < 0:\\
&p'_1 = [0: 1: 0:0:0:0:0],\, p'_2 = [0: 0: 0: 1 : 0 : 0 : 0], &p'_1 = [0: 0: 1:0:0:0:0],\, p'_2 = [0: 0: 1: 0 : 0 : 0 : 0],\\ 
&\Ch = a + b - (1+\delta)(a+b) = -\delta(a+b)< 0.&\Ch = 2(a+b)- (1+\delta)(a+b) = (1-\delta)(a+b) < 0.
\end{matrix*}
\end{align*}
\begin{align*}
\begin{matrix*}[l]
& a > 0, b < 0, a+b > 0: & a > 0, b < 0, a+b < 0:\\
& p'_1 = [0: 0: 1:0:0:0:0],\, p'_2 = [0: 0: 0: 1 : 0 : 0 : 0],& p'_1 = [0: 0: 1:0:0:0:0],\, p'_2 = [0: 0: 0: 1 : 0 : 0 : 0],\\ 
&\Ch = a+b + b - (1+\delta)(a+b) = -\delta a +(1 - \delta)b < 0. &\Ch = -\delta a +(1 - \delta)b < 0.
\end{matrix*}
\end{align*}
\begin{align*}
\begin{matrix*}[l]
&a < 0, b > 0, a+b>0: &a < 0, b > 0, a+b<0:\\
& p'_1 = [0: 1: 0:0:0:0:0],\, p'_2 = [0: 0: 1: 0 : 0 : 0 : 0],& p'_1 = [0: 1: 0:0:0:0:0],\, p'_2 = [0: 0: 1: 0 : 0 : 0 : 0],\\ 
&\Ch =a + a + b - (1+\delta)(a+b) = (1-\delta) a -\delta b < 0 &\Ch = (1-\delta) a -\delta b < 0.
\end{matrix*}
\end{align*}
\begin{align*}
\begin{matrix*}[l]
&a = 0, b > 0:& a = 0, b < 0:\\
& p'_1 = [0: 1: 0:0:0:0:0],\, p'_2 = [0: 0: 1: -e^{-a_2} : 0 : 0 : 0],& p'_1 = [0: 0: 1:0:0:0:0],\, p'_2 = [0: 0: 1: -e^{-a_2} : 0 : 0 : 0],\\ 
&\Ch = a + a + b - (1+\delta)(a+b) = -\delta b < 0. &\Ch = 2(a+b) - (1+\delta)(a+b) = (1-\delta) b < 0. 
\end{matrix*}
\end{align*}
\begin{align*}
\begin{matrix*}[l]
& a > 0, b = 0:& a < 0, b = 0:\\
& p'_1 = [0: 1: -e^{a_2}:0:0:0:0], p'_2 = [0: 0: 0: 1 : 0 : 0 : 0], & p'_1 = [0: 1: -e^{a_2}:0:0:0:0],\, p'_2 = [0: 0: 1: 0 : 0 : 0 : 0],\\ 
&\Ch = a + b - (1+\delta)(a+b) = -\delta a < 0.&\Ch = a + a + b - (1+\delta)(a+b) = (1-\delta)a < 0.
\end{matrix*}
\end{align*}
}
It follows that the auxiliary surface
\begin{equation*}
Z' = \Bl_{p_1, p_2, q'_2} S_6
\end{equation*}
admits a unique cscK metric in the class
\begin{equation*}
[\omega'] = \pi^* c_1(S_6) - \varepsilon^2_1 ([E_{p_1}] + [E_{p_2}] + \delta [E_{q'_2}]),
\end{equation*}
for $0 < \varepsilon_1 \ll 1$ and so, since $Z'$ has discrete automorphisms, the compactified LG model
\begin{equation*}
Z = \Bl_{p_3, p_4, q'_3} Z'
\end{equation*}
admits a unique cscK metric in the class
\begin{equation*}
[\omega] = \pi^* c_1(S_6) - \varepsilon^2_1 ([E_{p_1}] + [E_{p_2}] + \delta [E_{q'_2}]) - \varepsilon^2_2([E_{p_3}] + [E_{p_4}] + [E_{q'_3}])
\end{equation*}
for $0 < \varepsilon_2 \ll \varepsilon_1 \ll 1$. The parameters $\varepsilon_1, \varepsilon_2$ can be chosen uniformly as long as the basepoints corresponding to $[\omega_{\C}]$ remain bounded away from the torus fixed points of $T_{\nabla}$.
\section{$\Delta = P_{8b},\,P_{8c}$}
We give full details for $\Delta = P_{8b}$, the case when $\Delta = P_{8c}$ is similar. This is the mirror of the previous Section, so $T_X = T_{\Delta}$, $\Delta = P_{8b}$ is a Gorenstein toric degeneration of a smooth smooth degree $4$ del Pezzo $X = S_4$. Fix the complexified effective divisor $H = a_0 l + \sum^5_{i = 1} E_i$.

We proceed as in Section \ref{CubicSection}. Up to a reflection, the polygon $P_9$ is obtained from $P_6$ by adding to the ordered set of vertices
\begin{align*}
\{v_1 = (0,1),\,v_2 = (1,1),\,v_3 = (1,0),\,v_4 = (0,-1),\,v_5 = (-1,-1),\,v_6 = (-1,0)\}.
\end{align*}
the lattice points $\{ v_7 = (1,-1),\,v_8 = (-2,-1)\}$. So, according to \cite{Przyjalkowski_CYLG}, Section 3, Construction 16, the LG potential of the crepant resolution $\widetilde{T}_{\Delta}$ is given by
\begin{align*} 
& f_{\widetilde{T}_{\Delta}, H} = f_{S_6, H'} +  e^{-a_0-a_1} e^{-a_4} \frac{x}{y} + e^{-a_0} e^{-a_0 -a_3} e^{-a_5} \frac{1}{x^2 y}.
\end{align*}
The LG potential of $S_4$ is obtained by adding suitable corrections to $f_{\widetilde{T}_{\Delta}, H}$, due to the orbifold singularities. Such corrections correspond to ordered sets of lattice points on the boundary of $\Delta$,
\begin{align*}
\{(v_7, v_4, v_5, v_8),\,(v_2, v_3, v_7),\,(v_8, v_6, v_1)\}.
\end{align*}
Write
\begin{equation*}
f_{\widetilde{T}_{\Delta}, H} = \sum_i m_{v_i} x^{v_i}.
\end{equation*}
Then, $f_{S_{4}, H}$ is a sum of contributions from vertices and non-vertex lattice points on the boundary,  
\begin{align*}
& f_{S_{4}, H} =  y + e^{-a_2} x y + e^{-a_0-a_1} e^{-a_4} \frac{x}{y} + e^{-a_0} e^{-a_0 -a_3} e^{-a_5} \frac{1}{x^2 y} + \sum^3_{i=1}f^{(i)}_{S_4, H},
\end{align*}
where
\begin{align*} 
& f^{(1)}_{S_4, H} = \left(m_{v_7}\left(1+\frac{m_{v_4}}{m_{v_7}} s\right)\left(1+\frac{m_{v_5}}{m_{v_4}}s\right)\left(1+\frac{m_{v_8}}{m_{v_5}}s\right)\right)([s] x^{v_4} + [s^2] x^{v_5}),\\
&f^{(2)}_{S_4, H} = \left(m_{v_2}\left(1+\frac{m_{v_3}}{m_{v_2}} s\right)\left(1+\frac{m_{v_7}}{m_{v_3}}s\right)\right)([s] x^{v_3}),\\
&f^{(3)}_{S_4, H} = \left(m_{v_8}\left(1+\frac{m_{v_6}}{m_{v_8}} s\right)\left(1+\frac{m_{v_1}}{m_{v_6}}s\right)\right)([s] x^{v_6}).
\end{align*}
The dual of $P_{8b}$ is $P_{4b}$, so $T_{\nabla}$ is isomorphic to $\PP(\mathcal{O}(1)\oplus\mathcal{O})$. The equations cutting out the canonical embedding $T_{\nabla} \hookrightarrow \PP[x_0 : \cdots : x_8]$ are the same as in Section \ref{CubicSection} after eliminating all equations involving $x_7$ there and replacing $x_8, x_9$ there with $x_7, x_8$, and so are given by the quadrics \eqref{quadEqsP6}, new relations involving $x_0$, 
\begin{align*}
\{x_3 x_4 = x_0 x_7,\,x_5 x_6 = x_0 x_8\}, 
\end{align*}
as well as  
\begin{align*}
\{x_1 x_8 - x^2_6,\,x_2 x_7 - x^2_3,\,x_5 x_7 - x^2_4,\,x_4 x_8 - x^2_5,\,x_7 x_8 - x_4 x_5\}.
\end{align*}
Thus, the components of the toric boundary of $T_{\nabla}$ are given by the linear subspace and conics
\begin{align*}
&C_1 = \PP[x_1: x_2],\,C_2 = \{x_2 x_7 - x^2_3\} = \{[s^2:st:t^2]\}\subset\PP[x_2 : x_3 : x_7],\\ 
&C_3 = \{x_1 x_8 - x^2_6\} = \{[s^2:st:t^2]\}\subset \PP[x_1: x_6: x_8] 
\end{align*}
and the twisted cubic
\begin{align*}
C_4 = \{x_5 x_7 - x^2_4,\,x_4 x_8 - x^2_5,\,x_7 x_8 - x_4 x_5\}\cong \{[s^2 t : s t^2 : s^3 : t^3]\} \subset\PP[x_4: x_5: x_7: x_8].
\end{align*}
The upshot is that we can now compute explicitly $f_{S_{4}, H}$ and its restriction to the toric boundary of $T_{\nabla}$. We find (up to factors in $\C^*$)
\begin{align*}
& f_{S_{4}, H}|_{C_1} \propto e^{-a_2 }+ s,\\
& f_{S_{4}, H}|_{C_2} \propto \left(e^{ a_2 }+s\right) \left(s e^{ a_0 +  a_1 +  a_4 }+1\right),\\
& f_{S_{4}, H}|_{C_3} \propto \left(s e^{ a_0 +  a_1 }+1\right) \left(s e^{ a_0 + a_3 + a_5}+e^{ a_1 }\right),\\
& f_{S_{4}, H}|_{C_4} \propto \left(e^{ a_1} + s\right) \left(e^{a_4} + s\right) \left(s
   e^{ a_0 + a_3 + a_5}+1\right).
\end{align*}
In particular, the base locus of the pencil defined by $f_{S_4, H}$ is given by the cycle
\begin{align*}
\Bs(\fd_{\Delta}) = \{p_1, \ldots, p_8\} \subset T_{\nabla} \cong \PP(\mathcal{O}(1)\oplus\mathcal{O}) \subset \PP^8,
\end{align*}
where
\begin{align}\label{Deg4Arrangement}
\nonumber&p_1 = [e^{-a_2 } : 1] \in C_1 \subset \PP[x_1 : x_2],\\
\nonumber&p_2 = [e^{2 a_2}: - e^{a_2}: 1] \in C_2 \subset \PP[x_2 : x_3 : x_7],\\ 
\nonumber&p_3 = [e^{-2(a_0 +  a_1 + a_4)}: -e^{ -a_0 -  a_1 -  a_4 }:1]\in C_2 \subset  \PP[x_2 : x_3 : x_7],\\
\nonumber&p_4 = [s^{-2(a_0 +  a_1)} : -e^{-a_0 -  a_1} :1] \in C_3 \subset \PP[x_1: x_6: x_8],\\
\nonumber&p_5 = [e^{2(a_1 - a_0 - a_3 - a_5)} : -e^{a_1 - a_0 - a_3 - a_5} : 1] \in C_3 \subset \PP[x_1: x_6: x_8],\\ 
\nonumber&p_6 = [ e^{2a_1} : -e^{a_1}  : -e^{3a_1} : 1] \in C_4 \subset \PP[x_4: x_5: x_7: x_8],\\
\nonumber&p_7 = [ e^{2a_4} : -e^{a_4}  : -e^{3a_4} : 1] \in C_4 \subset \PP[x_4: x_5: x_7: x_8],\\
&p_8 = [e^{-2(a_0 + a_3 + a_5)} : -e^{-a_0 - a_3 - a_5} : -e^{-3(a_0 + a_3 + a_5)} : 1] \in C_4 \subset \PP[x_4: x_5: x_7: x_8].
\end{align}
As in Section \ref{CubicSection} we introduce the open loci
\begin{equation*}
\mathcal{D} := \{p_i = p_j \text{ for some } i \neq j \},\,\mathcal{D}_{> \delta} := \{ \operatorname{dist}([\omega_{\C}], \mathcal{D}) > \delta > 0\} \subset \Ka(S_4)_{\C}.
\end{equation*}
Then $\mathcal{D}$ is given by the intersection with $\Ka(S_4)_{\C}$ of the \emph{explicit, affine hyperplane arrangement} in $H^{1,1}(X, \C)$ determined by \eqref{Deg4Arrangement}. For fixed $\delta > 0$, the points $\{p_i\}$ is the support of the base locus corresponding to $[\omega_{\C}]\in\mathcal{D}_{>\delta}$ are uniformly separated.

Fix K\"ahler parameters $[\omega_{\C}]$ on $X$ lying in $\mathcal{D}_{>\delta}$. Using $T_{\nabla} \cong \PP(\mathcal{O}(1)\oplus\mathcal{O})$, $C_4$ is given by the zero section, $C_2, C_3$ by fibres, $C_1$ by the infinity section. So identifying $T_{\nabla} \cong \Bl_q \PP^2$ (as toric surfaces with their toric boundaries), we can present the compactified toric LG mirror $Z$ of $X$ as
\begin{equation*}
Z = \Bl_{p_1}\Bl_{q} Y,\,Y:=\Bl_{p_2, \ldots, p_8} \PP^2,
\end{equation*}
where $\{p_2, p_3\} \subset L_1$, $\{p_4, p_5\} \subset L_2$, $\{p_6, p_7, p_8\} \subset L_3$ are non-fixed points in torus invariant lines $L_i \subset \PP^2$, $q = L_1 \cap L_2$ is a torus fixed point, and $p_1$ is a non-fixed point in the exceptional divisor $E_q$. We can study the stability of the cycle $\sum^8_{i = 2} p_i$ as a point of the relevant Chow variety of $\PP^2$ using the Hilbert-Mumford Criterion. An elementary computation shows that this cycle is in fact stable (essentially because all $p_i$ are distinct and at most $3$ of the $p_i$, $i = 2, \ldots, p_8$ lie on a line). It follows that the class on $Y$ given by
\begin{equation*}
[\omega_Y] := \pi^*c_1(\mathcal{O}_{\PP^2}(1)) - \varepsilon^2_1 \sum^8_{i = 2} [E_{p_i}],\,0 < \varepsilon_1 \ll 1
\end{equation*} 
contains a unique cscK representative. In turn, since $Y$ has discrete automorphisms, we find that the class on the compactified LG model $Z$ given by
\begin{equation*}
[\omega_Z] := \pi^*c_1(\mathcal{O}_{\PP^2}(1)) - \varepsilon^2_1 \sum^8_{i = 2} [E_{p_i}] - \varepsilon^2_2 [E_q] - (\varepsilon^2_2 + \varepsilon^2_3) [E_{p_1}],\,0 <  \varepsilon_3 \ll \varepsilon_2 \ll \varepsilon_1 \ll 1 
\end{equation*}  
contains a unique cscK representative. As usual, by the analytic estimates, the parameters $\varepsilon_i$ can be chosen uniformly for $[\omega_{\C}] \in \mathcal{D}_{>\delta}$, so we obtain a map to a moduli space $\mathfrak{M}$ containing $K$-stable polarised rational elliptic surfaces.

The extension to non-generic strata, containing the complex structure mirror to $c_1(X)$, can be handled just as in Sections \ref{C1StrataSec}-\ref{NonGenericSec}, by the polystability of the cycle $p_2 + p_4 + p_6$.

\section{$X = T_{P_{4a}}$}
Consider $X = T_{\Delta} \cong \PP^1 \times \PP^1$ corresponding to the polygon $\Delta = P_{4a}$. Fix a complexified effective divisor $H = (a, b)$ on $X$. The corresponding LG potential is 
\begin{equation*}
f = x + \frac{e^{-a}}{x} + y + \frac{e^{-b}}{y}. 
\end{equation*}
The dual of $\Delta$ is $\nabla = P_{8a}$. The anticanonical embedding of $T_{\nabla} \subset \PP[x_0 : \cdots : x_4]$ is cut out by the quadrics
\begin{align*}
\{x_1 x_3 = x^2_0,\,x_2 x_4 = x^2_0\}
\end{align*}
corresponding to the monoid relations among the vertices of $P_{4a}$,
\begin{align*}
\{v_1 = (0,1),\,v_2 = (1,0),\,v_3 = (0,-1),\,v_4 = (-1,0)\}.
\end{align*}
The anticanonical degree of $T_{\nabla}$ is given by $\vol \Delta = 4$. Thus, $T_{\nabla}$ is a Gorenstein toric degeneration of any degree $4$ del Pezzo $S_4$. The LG potential $f$ extends to a rational anticanonical pencil on $T_{\nabla}$ given by
\begin{equation*}
\mathfrak{d}_{\Delta} = |x_1 + x_2 + e^{-b} x_3 + e^{-a} x_4, x_0|.
\end{equation*}
The irreducible components of the toric boundary $(x_0)$ of $T_{\nabla}$ are
\begin{align*}
C_1 = \PP[x_1: x_2],\, C_2 = \PP[x_1: x_4],\,C_3 = \PP[x_3: x_4],\,C_4 = \PP[x_2:x_3]. 
\end{align*}
Thus, the base locus $\Bs(\mathfrak{d}_{\Delta})$ is given by
\begin{align*}
&p_1 = [1:-1] \in C_1 = \PP[x_1: x_2],\,p_2 = [1: - e^a] \in C_2 = \PP[x_1: x_4],\\
&p_3 = [1: -e^{a-b}] \in C_3 = \PP[x_3: x_4],\,p_4 = [1: -e^b] \in C_4 = \PP[x_2:x_3].
\end{align*}
The barycenter of the reflexive polygon $P_{8a}$ vanishes, so according to \cite{Zhu_orbisolitons}, $T_{\nabla}$ admits a K\"ahler-Einstein orbifold metric $\omega_{\nabla}$. The complexification $G = K^{\C}$ of the Hamiltonian isometry group $K$ of $\omega_{\nabla}$ is isomorphic to the torus $(\C^*)^2$ acting by 
\begin{equation*}
(\lambda, \nu)\cdot [x_0:x_1:x_2:x_3:x_4] = [x_0:\lambda x_1: \nu x_2: \lambda^{-1} x_3: \nu^{-1} x_4].
\end{equation*}
Thus, a 1PS acts by 
\begin{equation*}
t\cdot [x_0:x_1:x_2:x_3:x_4] = [x_0: t^r x_1: t^s x_2: t^{-r} x_3: t^{-s} x_4].
\end{equation*}
The induced action on $\Bs(\mathfrak{d}_{\Delta})$ is 
\begin{align*}
&t\cdot p_1 = [0: t^r:-t^s:0:0],\,t\cdot p_2 = [0: t^{r+s} : 0: 0: - e^a ],\\
&t\cdot p_3 = [0: 0:0:t^{s}: -t^{r} e^{a-b}],\,t\cdot p_4 = [0: 0:t^{r+s}:-e^b:0]. 
\end{align*}
The Hilbert-Mumford computation can be summarised as follows. 
{\tiny 
\begin{align*}
\begin{matrix*}[l]
&r > s,\,r+s>0,&r > s,\,r+s < 0,\\
&p'_1 = [0: 0:1:0:0],\,p'_2 = [0: 0: 0: 0: 1],& p'_1 = [0: 0:1:0:0],\, p'_2 = [0: 1 : 0: 0: 0],\\
&p'_3 = [0: 0:0:1: 0],\, p'_4 = [0: 0:0:1:0],& p'_3 = [0: 0:0:1: 0],\, p'_4 = [0: 0:1:0:0],\\
&\Ch = (s-r) - (r + s) < 0. &\Ch =   (s-r)+(r+s) < 0.\\
\end{matrix*} 
\end{align*}
\begin{align*}
\begin{matrix*}[l]
&r < s,\,r+s > 0,&r < s,\,r+s < 0,\\
& p'_1 = [0: 1:0:0:0],\, p'_2 = [0: 0 : 0: 0: 1],& p'_1 = [0: 1:0:0:0],\, p'_2 = [0: 1: 0: 0: 0 ],\\
& p'_3 = [0: 0:0:0: 1],\, p'_4 = [0: 0:0:1:0],& p'_3 = [0: 0:0:0: 1],\, p'_4 = [0: 0:1:0:0],\\
&\Ch =   (r-s)-(r+s)<0. &\Ch =   (r-s)+(r+s)<0. 
\end{matrix*} 
\end{align*}
\begin{align*}
\begin{matrix*}[l]
&r = s > 0,&r = s < 0,\\
& p'_1 = [0: 1:-1:0:0],\, p'_2 = [0: 0 : 0: 0: 1],& p'_1 = [0: 1:-1:0:0],\, p'_2 = [0: 1 : 0: 0: 0],\\
& p'_3 = [0: 0:0:1: - e^{a-b}],\, p'_4 = [0: 0:0:1:0],& p'_3 = [0: 0:0:1: - e^{a-b}],\, p'_4 = [0: 0: 1:0:0],\\ 
&\Ch = -2r < 0. &\Ch = 2r < 0.
\end{matrix*} 
\end{align*}
}Thus, $\Bs(\mathfrak{d}_{\Delta})$ defines a stable point in the relevant Chow variety of $T_{\nabla}$, with respect to the linearisation induced by $-K_{T_{\nabla}}$. It follows that the blowup $Y := \Bl_{\Bs(\mathfrak{d}_{\Delta})} T_{\nabla}$ admits a cscK representative in the class
\begin{equation*}
[\omega_Y] := \pi^*c_1(T_{\nabla}) - \varepsilon^2 \sum^4_{i = 1} [E_{p_i}],\,0 < \varepsilon \ll 1.
\end{equation*}
Moreover, $\varepsilon  > 0$ can be chosen uniformly as long as $\Bs(\mathfrak{d}_{\Delta})$ is bounded away from the orbifold points of $T_{\nabla}$. The compactified LG model $Z$ for $X$ is given by the crepant resolution of the orbifold singularities of $Y$. Since $Y$ has discrete automorphisms, by \cite{ArezzoPacard}, Theorem 1.3, the crepant resolution $Z$ admits a unique cscK metric in each K\"ahler class
\begin{equation*}
[\omega_{Z}] = \pi^*[\omega_{Y}] + \delta^2 \sum^3_{i=1}  [\tilde{\eta}_i],\, \delta \in (0, \bar{\delta})\text{ for some } \bar{\delta} > 0,
\end{equation*}  
where $[\tilde{\eta}_i]$, for $i=1, \ldots, 4$, is the image in $Z$ of the K\"ahler class $[\eta]$ of an ALE Ricci flat resolution of $\C^2/(\Z/2)$ with vanishing ADM mass.
\section{$T_X = T_{P_{8a}}$}\label{P8aSec}
We consider the toric mirror LG model for a smooth degree $4$ del Pezzo $S_4$ degenerating to the Gorenstein toric del Pezzo $T_X = T_{\Delta}$, $\Delta = P_{8a}$. The polygon $P_{8a}$ may be obtained from (a reflection of) $P_{6}$ by adding by adding to the ordered set of vertices
\begin{align*}
\{v_1 = (0,1),\,v_2 = (1,1),\,v_3 = (1,0),\,v_4 = (0,-1),\,v_5 = (-1,-1),\,v_6 = (-1,0)\}.
\end{align*}
the lattice points $\{ v_7 = (-1,1),\,v_8 = (1,-1)\}$. Thus, we have 
\begin{align*} 
& f_{\widetilde{T}_{\Delta}, H} = f_{S_6, H'} + e^{-a_0 - a_3 - a_4} \frac{y}{x} + e^{-a_0-a_1-a_5}\frac{x}{y}.
\end{align*}
Consider the positively (clockwise) ordered sets of lattice points in each boundary component, 
\begin{align*}
(w_0, w_1, w_2) = (v_7, v_1, v_2),\,(v_2, v_3, v_8),\,(v_8, v_4, v_5),\,(v_5, v_6, v_7).
\end{align*}
Then, setting
\begin{equation*}
f_{\widetilde{T}_{\Delta}, H} = \sum_i m_{v_i} x^{v_i},
\end{equation*}
the coefficient of $f_{S_4, H}$ at $w_1$ is given by the coefficient of $s$ in the polynomial
\begin{equation*} 
m_{w_0}\left(1+\frac{m_{w_1}}{m_{w_0}} s\right)\left(1+\frac{m_{w_2}}{m_{w_1}}s\right).
\end{equation*} 
The dual $\nabla$ to $\Delta$ is given by $\nabla = P_{4a}$. The Laurent polynomial $f_{S_4, H}$ extends to an anticanonical pencil $\mathfrak{d}$ on $T_{\nabla} \cong \PP^1 \times \PP^1$. The anticanonical map $T_{\nabla} \hookrightarrow \PP[x_0:\cdots:x_8]$ is cut out by the quadrics \eqref{quadEqsP6} together with the relations
\begin{align*}
&\{x_2 x_7 = x^2_1,\,x_2 x_8 = x^2_3,\,x_5 x_8 = x^2_4,\,x_5 x_7 = x^2_6\},\,\{x_7 x_8 = x_0^2,\,x_1 x_6 = x_0 x_8,\,x_3 x_4 = x_0 x_7\}.
\end{align*}
Thus, the toric boundary of $T_{\nabla}$ is given by the irreducible components
\begin{align*}
& C_1 := \{x_2 x_7 = x^2_1\} = [s t: s^2: t^2]\subset \PP[x_1:x_2: x_7],\\
&C_2 := \{x_2 x_8 = x^2_3\} = [s^2: s t: t^2]\subset \PP[x_2:x_3:x_8],\\
& C_3 := \{x_5 x_8 = x^2_4\} = [s t: s^2: t^2]\subset \PP[x_4:x_5: x_8],\\
&C_4:=\{x_5 x_7 = x^2_6\} = [s^2: s t: t^2] \subset \PP[x_5:x_6:x_7],
\end{align*}
and the restriction of the pencil $\mathfrak{d}_{\Delta}$ to the boundary components is given by
\begin{align*}
&f_{S_4, H}|_{C_1} \propto \left(e^{a_2}+s\right) \left(s e^{a_0+a_3+a_4}+1\right),\,f_{S_4, H}|_{C_2} \propto \left(e^{a_2}+s\right) \left(s e^{a_0+a_1+a_5}+1\right),\\
&f_{S_4, H}|_{C_3} \propto \left(e^{a_3} s+1\right) \left(s e^{a_1+a_5}+e^{a_3}\right),\,f_{S_4, H}|_{C_4} \propto \left(e^{ a_1 } s+1\right) \left(e^{a_1}+s e^{a_3+a_4}\right).
\end{align*}
The base locus $\Bs(\mathfrak{d}_{\Delta})$ is  
\begin{align*}
&p_1 = [-e^{a_2}:e^{2a_2}:1],\,p_2 = [-e^{-a_0-a_3-a_4}:e^{-2(a_0+a_3+a_4)} :1] \in C_1 \subset \PP[x_1:x_2: x_7],\\ 
&p_3 = [e^{2a_2}:-e^{a_2}:1],\,p_4 =[e^{-2(a_0+a_1+a_5)}:-e^{-a_0-a_1-a_5}:1] \in C_2 \subset \PP[x_2:x_3:x_8],\\
&p_5 = [-e^{-a_3}: e^{-2a_3}:1],\,p_6 = [-e^{-a_1+ a_3-a_5}: e^{-2(a_1- a_3+a_5)} :1] \in C_3 \subset \PP[x_4:x_5: x_8],\\
&p_7 =[e^{-2a_1 }: -e^{-a_1 }:1],\,p_8 =[e^{2(a_1-a_3-a_4)}:-e^{a_1-a_3-a_4} :1] \in C_4 \subset \PP[x_5:x_6:x_7]. 
\end{align*}
Using the anticanonical embedding, we identify each $p_i$ with a point in the toric boundary of $\PP^1 \times \PP^1$ (with a fixed toric structure) which is not torus-fixed. Consider the polystability of the cycle $\sum^8_{i=1} p_i$ under the action of $\Aut_0(\PP^1\times\PP^1) \cong \PP SL(2, \C) \times \PP SL(2, \C)$, with respect to the polarisation $(1, 1)$. In suitable coordinates, a 1PS acts by 
\begin{equation*}
t \cdot ([u_0 : u_1],[z_0: z_1]) = ([\lambda^a u_0 : \lambda^{-a} u_1],[\lambda^{b}z_0: \lambda^{-b}z_1]),\,a \geq b \geq 0.
\end{equation*}     
So, the only possible positive contributions to the Chow weight of the cycle come from points lying on the divisors $u_1 = 0$ or $z_1 = 0$, and we can estimate, for a nontrivial 1PS
\begin{equation*}
\Ch \leq (a+ b) + 2 (a-b) + (-a+b) - 4(a+b) < 0.  
\end{equation*}  
It follows that for $[\omega_{\C}] \in \mathcal{D}_{> \delta}$ the classes on the compactified toric LG model $Z:= \Bl_{\Bs(\mathfrak{d}_{\Delta})}(T_{\nabla})$ given by
\begin{equation*}
[\omega_Z] = \pi^*c_1(\mathcal{O}_{T_{\nabla}}(1,1)) - \varepsilon^2 \sum^8_{i=1}[E_i],\,0<\varepsilon \ll 1
\end{equation*}
is cscK (with $\varepsilon_i$ depending only on $\delta > 0$). A similar analysis shows that the classes on $Z$ given by
\begin{equation*}
[\omega'_Z] = \pi^*\mathcal{O}_{T_{\nabla}}(1,1) -  \varepsilon^2_1 [E_1 + E_3 + E_5 + E_7] - \varepsilon^2_2 [E_2 + E_4 + E_6 + E_8],\,0 < \varepsilon_2 \ll \varepsilon_1 \ll 1  
\end{equation*} 
are also cscK. By our deformation theory argument, for sufficiently small $\delta > 0$, the generic set $\mathcal{D}_{>\delta}$ has nonempty intersection with an open neighbourhood of $c_1(S_4) \in \Ka(S_4)_{\C}$ corresponding to cscK complex structures on $Z$. 

\section{$T_X = T_{P_{4c}}$}\label{QuadraticConeSec}

The toric surface $T_X$ with polytope $P_{4c}$ is a quadratic cone, the unique nontrivial Gorenstein toric degeneration of $X \cong \PP^1 \times \PP^1$. Up to a reflection, the corresponding LG potential for $X$ with respect to $H = (a, b)$ is given in \cite{Przyjalkowski_CYLG}, Example 24,
\begin{equation*}
f = y + \frac{e^{-a}}{x y} + \frac{e^{-a} + e^{-b}}{y} + e^{-b}\frac{x}{y}.
\end{equation*}
The potential $f$ extends to a rational pencil on the dual variety $T_{\nabla} = T_{P_{8c}}$ given by
\begin{equation*}
\mathfrak{d}_{\Delta} = |x_1 +  e^{-a} x_4 + (e^{-a} + e^{-b}) x_3 + e^{-b}x_2 , x_0 |,
\end{equation*}  
where the anticanonical embedding $T_{\nabla} \hookrightarrow \PP^4$ is cut out by 
\begin{align*}
\{x_1 x_3 = x^2_0,\, x_2 x_4 = x^2_3\}.
\end{align*}
Thus, the toric boundary components are given by
\begin{align*}
C_1 = \PP[x_1 : x_4],\,C_2 = \PP[x_1:x_2],\,C_3 = \{x_2 x_4 = x^2_3\} \subset\PP[x_2:x_3:x_4],
\end{align*}
and the base locus $\Bs(f)$ is  
\begin{align*}
& p_1 = [1 : - e^a]\in C_1,\,p_2 = [1: -e^b] \in C_2,\\
& p_3 = [1:-1:1],\,p_4 = [1:-e^{a-b}:e^{2(a-b)}] \in C_3.
\end{align*}
By considering a maximal subdivision of $\nabla$, we find an isomorphism for the crepant resolution 
\begin{equation*}
\widetilde{T}_{\nabla} \cong \Bl_{q'_1, q'_2}\Bl_{q_1, q_2}\PP^1\times\PP^1, 
\end{equation*}
where $q_i$ are torus-fixed points contained in the same fibre, and $q'_i$ denotes a torus-fixed point in the exceptional divisor over $q_i$. Thus, the compactified toric LG model $Z$ is given by
\begin{equation*}
Z = \Bl_{p_1, p_2, p_3, p_4} \Bl_{q'_1, q'_2}\Bl_{q_1, q_2}\PP^1\times\PP^1 \cong \Bl_{q'_1, q'_2} \Bl_{p_1, p_2, p_3, p_4} \Bl_{q_1, q_2}\PP^1\times\PP^1 
\end{equation*}
where the isomorphisms can be chosen so that 
\begin{align*}
& q_1 = ([1: 0], [1:0]),\,q_2 = ([1:0],[0:1]),\, p_1 = ([p_{11}:p_{12}]:[1:0])\\
& p_2 = ([p_{21}:p_{22}]:[0:1]),\,p_3 = ([0:1], [p_{31}:p_{32}]),\,p_4 = ([0:1],[p_{41}:p_{42}]). 
\end{align*}
For $a \neq b$, the cycle $q_1 + q_2 + \sum^4_{i = 1} p_i$ is polystable with respect to the linearised action of $SL(2,\C)\times SL(2, \C)$, and after its blowup the residual automorphisms are discrete. The extension across the special locus $a = b$ can be handled similarly by using the polystability of the cycle $q_1 + q_2 + p_1 + p_2 + p_3$. 

\section{$T_{X} = T_{P_{5a}}$} 

The smooth toric del Pezzo $X = T_{X} = T_{P_{5a}}$ is isomorphic to $\Bl_{r_1} \PP^1 \times \PP^1 \cong \Bl_{r_1, r_2} \PP^2$, i.e. $X \cong S_7$ is a degree $7$ del Pezzo. According to \cite{Przyjalkowski_CYLG}, Example 24, the LG potential of the trivial degeneration of $S_7$ with respect to $H = a_0 l + a_1 E_1 + a_2 E_2$ is given by
\begin{equation*}
f = x + y + e^{-a_0} \frac{1}{x y} + e^{-a_0-a_1} \frac{1}{ y} + e^{-a_2} x y 
\end{equation*}
(this agrees with $P_{5a}$ up to $GL(2, \Z)$). The corresponding anticanonical pencil on $T_{\nabla} = T_{P_{7a}}$ is given by
\begin{equation*}
\mathfrak{d}_{\Delta} = | x_1 + x_3 + e^{-a_0} x_5 + e^{-a_0-a_1} x_4 + e^{-a_2} x_2, x_0 |,
\end{equation*}  
where the anticanonical embedding $T_{\nabla} \hookrightarrow \PP^5$ is cut out by the quadratic equations obtained from \eqref{quadEqsP6} by omitting all relations involving the variable $x_4$ there and replacing $x_6$ there with the current variable $x_4$. So, the boundary components are $C_i \cong \PP[x_i : x_{i+1}]$ for $i = 1, \ldots, 5$, and $\Bs(f)$ is given by 
\begin{align*}
& p_1 = [1 : - e^{a_2}] \in C_1,\,p_2 = [1: - e^{-a_2}] \in C_2,\, p_3 = [1 : - e^{a_0 + a_1}] \in C_3,\\
& p_4 = [1: -e^{-a_1}] \in C_4, \, p_5 = [1: -e^{a_0}] \in C_5.
\end{align*}
Looking at the maximal subdivision of $\nabla$ we find an isomorphism 
\begin{equation*}
\widetilde{T}_{\nabla} \cong \Bl_{q_1, q_2, q_3} \PP^1 \times \PP^1,
\end{equation*}
where the $q_i$ are torus-fixed points. Moreover, we see that the isomorphism can be chosen so that the basepoints $p_1$, $p_2$ are mapped to intersecting toric boundary components of $\PP^1 \times \PP^1$. The compactified toric LG model can be presented as
\begin{equation*}
Z = \Bl_{p_3, p_4, p_5} \Bl_{p_1, p_2, q_1, q_2, q_3} \PP^1 \times \PP^1,
\end{equation*}
and the cycle $p_1 + p_2 + \sum^3_{i=1} q_i$ is polystable, with discrete residual automorphisms.
\section{$T_X = T_{5b}$}
The surface $T_X = T_{5b}$ is the unique nontrivial Gorenstein toric degeneration of $X = S_7$. According to \cite{Przyjalkowski_CYLG}, Example 24, the corresponding LG potential with respect to $H = a_0 l + a_1 E_1 + a_2 E_2$ is given by
\begin{equation}\label{5bLGPotential}
f = x + y + \frac{e^{-a_0}}{x y} + \big(e^{-(a_0+ a_1)} + e^{-(a_0 + a_2)}\big)\frac{1}{y} + e^{-(a_0 + a_1 + a_2)}\frac{x}{y} 
\end{equation} 
(this agrees with $P_{5b}$ up to a reflection). The anticanonical pencil on $T_{\nabla} = T_{P_{7b}}$ induced by $f$ is given by
\begin{equation*}
\mathfrak{d}_{\Delta} = | x_2 + x_1 + e^{-a_0} x_5 + (e^{-(a_0+ a_1)} + e^{-(a_0 + a_2)})x_4 + e^{-(a_0 + a_1 + a_2)}x_3, x_0 |,
\end{equation*}  
where the anticanonical embedding $T_{\nabla} \hookrightarrow \PP^5$ is cut out by 
\begin{align*}
\{x_1 x_4 = x^2_0,\, x_3 x_5 = x^2_4,\,x_1 x_3 = x_0 x_2,\, x_2 x_5 = x_0 x_4\}.
\end{align*}
The toric boundary is  
\begin{align*}
C_1 = \{x_3 x_5 = x^2_4\} \subset \PP[x_3 : x_4: x_5],\,C_2 = \PP[x_1 : x_2],\,C_3 = \PP[x_1 : x_5],\,C_4 = \PP[x_2:x_3],
\end{align*}
so $\Bs(f)$ is given by
\begin{align*}
&p_1 = [1: -e^{-a_1}: e^{-2 a_1}],\, p_2 = [1: -e^{-a_2}: e^{-2 a_2}]\in C_1,\\
&p_3 = [1:-1]\in C_2,\,p_4 = [1:-e^{a_0}] \in C_3,\,p_5 = [1: -e^{a_0 + a_1 + a_2}] \in C_4.
\end{align*}
Considering a maximal subdivision of $\nabla$, we find an isomorphism
\begin{equation*}
\widetilde{T}_{\nabla} \cong \Bl_{q'_2}\Bl_{q_1, q_2} \PP^1 \times \PP^1,
\end{equation*}
where $q_1, q_2$ are torus-fixed points in the same fibre, and $q'_2$ is a torus-fixed point on the exceptional divisor over $q_2$. Moreover, the isomorphism can be chosen so that the basepoints $p_4$, $p_5$ are mapped to intersecting toric boundary components of $\PP^1 \times \PP^1$, one of them not containing $q_1, q_2$. Thus, the cycle $q_1 + q_2 + p_4 + p_5$ is polystable, with discrete residual automorphisms, and we can apply this to the compactified toric LG model
\begin{equation*}
Z = \Bl_{p_1, p_2, p_3} \Bl_{q'_2}\Bl_{p_4, p_5, q_1, q_2} \PP^1 \times \PP^1.
\end{equation*}
\section{$\Delta = P_{6b},\,P_{6c},\,P_{6d}$}
The reflexive polygons $P_{6b},\,P_{6c},\,P_{6d}$ are self-polar (up to $GL(2, \Z)$). We only consider $\Delta = P_{6b}$, the other cases are similar. 

We choose presentations such that $P_{5b}$ is the Newton polygon of \eqref{5bLGPotential}, while $P_{6b}$ is obtained from $P_{5b}$ by adding the lattice point $v_6 = (1, 1)$. Thus, the auxiliary LG potential of the weak del Pezzo $\widetilde{T}_{\Delta}$, with respect to $H = a_0 l + \sum^3_{i = 1} a_i E_i$, is given by
\begin{equation}\label{6bLGPotential}
f_{\widetilde{T}_{\Delta}} = x + y + \frac{e^{-a_0}}{x y} +  e^{-(a_0+ a_1)}  \frac{1}{y} + e^{-(a_0 + a_1 + a_2)}\frac{x}{y} + e^{-a_3} x y,
\end{equation}
while the actual LG potential of a degree $6$ del Pezzo $X$ degenerating to $T_{\Delta}$ is the rational pencil 
\begin{align*}
&f_{X} = (1 + e^{-(a_0 + a_1 + a_2 + a_3)}) x_2 + x_1 +  e^{-a_0} x_5 +  (e^{-(a_0+a_1)} + e^{-(a_0 + a_2)} ) x_4 \\
&+ e^{-(a_0 + a_1 + a_2)}x_3 + e^{-a_3} x_6,
\end{align*}
where the toric boundary is given by 
\begin{align*}
&C_1 = \PP[x_1 : x_6],\,C_2 = \{x_3 x_6 = x^2_2\} \subset \PP[x_2:x_3:x_6],\\
&C_3 = \{x_3 x_5 = x^2_4\} \subset \PP[x_3:x_4:x_5],\,C_4 = \PP[x_1:x_5].
\end{align*}  
So, the base locus is 
\begin{align*}
& p_1 = [1 : -e^{a_3}] \in C_1,\, \, p_2 = [ -e^{-(a_0 + a_1+ a_2)}: 1 : e^{-2(a_0+ a_1 + a_2)}] \in C_2,\\
& p_3 = [-e^{a_3}:1: e^{2a_3}] \in C_2,\,p_4 = [1: -e^{-a_1}: e^{- 2a_1}] \in C_3,\\
& p_5 = [1: -e^{-a_2}: e^{-2a_2}] \in C_3,\, p_6 = [1: -e^{a_0}] \in C_4.
\end{align*}
Considering a maximal subdivision of $\nabla$, which is $GL(2, \Z)$-equivalent to $\Delta$, we find an isomophism 
\begin{equation*}
\widetilde{T}_{\nabla} \cong \Bl_{q'_2} \Bl_{q_1, q_2}\PP^2,
\end{equation*}
where $q_1, q_2$ are torus-fixed points and $q'_2$ is a torus-fixed point in the exceptional divisor over $q_2$. Moreover, we can choose the isomorphism so that the basepoints $p_1, p_6$ are mapped to toric boundary components of $\PP^2$ in the fixed toric structure, distinct from the line through $q_1, q_2$. Then, the cycle $q_1 + q_2 + p_1 + p_6$ is polystable, with discrete residual automorphisms, and we can apply this to the compactified toric LG model
\begin{equation*}
Z = \Bl_{p_2, p_3, p_4, p_5} \Bl_{q'_2}\Bl_{p_1, p_6, q_1, q_2} \PP^2.
\end{equation*}
\section{$\Delta = P_{7a},\,P_{7b}$}\label{7aSec}
We only give some details for $\Delta = P_{7a}$, the other case is similar. We can choose presentations such that $P_{6a}$ is the Newton polygon of \eqref{6bLGPotential}, while $P_{7a}$ is obtained from $P_{6a}$ by adding the point $v_7 = (-1,0)$. The LG potential of the weak del Pezzo $\widetilde{T}_{\Delta}$, with respect to $H = a_0 l + \sum^4_{i = 1} a_i E_i$, is given by
\begin{equation*} 
f_{\widetilde{T}_{\Delta}} = x + y + e^{-a_0}\frac{1}{x y} +  e^{-(a_0+ a_1)}  \frac{1}{y} + e^{-(a_0 + a_1 + a_2)}\frac{x}{y} + e^{-a_3} x y + e^{-(a_0+a_4)} \frac{1}{x},
\end{equation*}
while the actual LG potential of a degree $5$ del Pezzo $X$ degenerating to $T_{\Delta}$ is the rational pencil 
\begin{align*}
&f_{X} = (1 + e^{-(a_0 + a_1 + a_2 + a_3)}) x_2 + x_1 +  e^{-a_0} x_5 +  (e^{-(a_0+a_1)} + e^{-(a_0 + a_2)} ) x_4 \\
&+ e^{-(a_0 + a_1 + a_2)}x_3 + e^{-a_3} x_6 + e^{-(a_0+a_4)} x_7.
\end{align*}
The dual variety is $T_{\nabla} = T_{P_{5a}} \cong \Bl_{r} \PP^1 \times \PP^1$, where $r$ is a torus-fixed point. The toric boundary is  
\begin{align*}
&C_1 = \PP[x_1 : x_6],\,C_2 = \{x_3 x_6 = x^2_2\} \subset \PP[x_2:x_3:x_6],\\
&C_3 = \{x_3 x_5 = x^2_4\} \subset \PP[x_3:x_4:x_5],\,C_4 = \PP[x_5:x_7],\,C_5 = \PP[x_1:x_7],
\end{align*}  
and the base locus  
\begin{align*}
& p_1 = [1 : -e^{a_3}] \in C_1,\, \, p_2 = [ -e^{-(a_0 + a_1+ a_2)}: 1 : e^{-2(a_0+ a_1 + a_2)}] \in C_2,\\
& p_3 = [-e^{a_3}:1: e^{2a_3}] \in C_2,\,p_4 = [1: -e^{-a_1}: e^{- 2a_1}] \in C_3,\\
& p_5 = [1: -e^{-a_2}: e^{-2a_2}] \in C_3,\, p_6 = [1: -e^{a_4}] \in C_4,\,p_7 = [1: -e^{a_0+a_4}] \in C_5.
\end{align*}
We can present the compactified toric LG model as
\begin{equation*}
Z = \Bl_{p_7} \Bl_r \Bl_{p_1, p_2, p_3, p_4, p_5, p_6} \PP^1 \times \PP^1,
\end{equation*}
and the cycle $r + \sum^6_{i = 1} p_i$ is polystable with respect to $SL(2, \C) \times SL(2, \C)$, with discrete residual automorphisms. The extension across the non-generic locus can be handled by the polystability of the cycle $r + p_1 + p_3 + p_5 + p_6$.
\section{Quartic threefold $V_4$}\label{Quartic3foldSec}
In this Section we begin the case by case analysis which leads to the proof of Theorem \ref{MainThmFanos}, following the outline given in Section \ref{3foldsSecIntro}. This is completed in Section \ref{1111Sec}.

Let $X = V_4$ be the Fano given by a smooth quartic threefold. Then $V_4$ admits a Gorenstein toric degeneration to a toric quartic $T_X \cong T_{\Delta}$, where $\Delta$ is the reflexive polytope $P_{4311}$ in the standard Kreuzer-Skarke list \cite{KreuzSkarke}, starting from $P_0$, the standard fan polytope of $\PP^3$. According to \cite{Ilten_weakLGmodels}, the corresponding toric LG model of $(V_4, 0)$ (i.e. $V_4$ endowed with the polarisation $c_1(V_4)$) is given (up to $GL(2, \Z)$ and shifts by a constant) by the Minkowski Laurent polynomial
\begin{align*}
f = \frac{(x + y + z + 1)^4}{x y z}.
\end{align*}
The polar dual $\nabla = \Delta^{\circ}$ is indeed given by $P_0$, so $f$ extends to an anticanonical pencil on $T_{\nabla} \cong \PP^3 = \PP[z_0: z_1: z_2 :z_3]$. In this case, it is easy to see this extension directly, 
\begin{equation*}
\mathfrak{d}_{\Delta} = |(z_0 + z_1 + z_2 + z_3)^4,  z_0 z_1 z_2 z_3|.
\end{equation*}
The base locus $\Bs(f)$ is contained in the toric boundary of $T_{\nabla}$. The general result \cite{Przyjalkowski_CYLG}, Lemma 25 says that, in dimension 3, for each toric stratum, the base locus is given by a union of transverse smooth rational curves, possibly with multiplicities. In the present case this is obvious, as the components $C_i$, $i = 0, \ldots, 3$ are given by
\begin{align*}
&C_0 = \{(z_1 + z_2 + z_3)^4 = 0\} \subset \PP[z_1 : z_2: z_3],\,C_1 = \{(z_0 + z_2 + z_3)^4 = 0\} \subset \PP[z_0 : z_2: z_3],\\ 
&C_2 = \{(z_0 + z_1 + z_3)^4 = 0\} \subset \PP[z_0 : z_1: z_3],\,C_3 = \{(z_0 + z_1 + z_2)^4 = 0\} \subset \PP[z_0 : z_1: z_2].
\end{align*} 
The intersection points of the reduced components, i.e. the nodes of the reduced base locus $(\Bs(f))^{\operatorname{red}}$, are given by
\begin{align*}
&p_1 = C_0 \cap C_1 = [0 : 0 : 1 : -1],\,p_2 = C_0 \cap C_2 = [0:1:0:-1],\\ 
&p_3 = C_0 \cap C_3 = [0: 1: -1:0],\,p_4 = C_1\cap C_2 = [1:0:0:-1]\\
&p_5 = C_1\cap C_3 = [1:0:-1:0],\,p_6 = C_2 \cap C_3 = [1:-1:0:0].
\end{align*}
According to \cite{Przyjalkowski_CYLG}, Proposition 26, a log Calabi-Yau compactified LG model $Z$ can be obtained by resolving the base locus of $\mathfrak{d}_{\Delta}$ by iterated blowups of its irreducible components in some fixed order. So the construction of $Z$ is not canonical.  

Moreover, at least when $T_{\nabla}$ is smooth, we also have the freedom of replacing $T_{\nabla}$ by its blowup along (some or all of the) $1$-dimensional strata of the toric boundary $D$ of $T_{\nabla}$, while replacing $D$ with $\pi^*D - E$. In the terminology of \cite{GrossHackingKeel_LCY} we are performing an admissible change of the toric model. Of course, we also have a similar freedom in dimension $2$ (by blowing up smooth torus fixed points), but there we have a canonical (minimal) choice of the toric model. 

In the present case, as a first step we blow up the $1$-dimensional toric stratum given by
\begin{equation*}
L_1 := \{z_0 = z_1 =0\} \cup L_2:=\{z_2 = z_3 = 0\}. 
\end{equation*}
So we replace $T_{\nabla}$ with $Z_1 = \Bl_{L_1} T_{\nabla}$ and the toric boundary and base locus with their proper transforms $D_1$, $\Bs(f)$ (with a standard abuse of notation for the latter). This new triple $(Z_1, D_1, f)$ still satisfies the assumptions of \cite{Przyjalkowski_CYLG}, Proposition 26. The base locus now contains the proper transforms of $C^{\operatorname{red}}_i$ together with the fibres of $E$ over the points $p_1$, $p_6$, the latter new components having strictly lower multiplicity. 

Note that $Z_1$ is isomorphic to the blowup of $\PP^2$ in disjoint lines or alternatively to the projectivised bundle $\PP(\mathcal{O} \oplus \mathcal{O}(1, -1))$ over $\PP^1 \times \PP^1$. Thus, $Z_1$ is K\"ahler-Einstein, i.e. the class
\begin{equation*}
\pi^*[\omega_{T_{\nabla}}] - [E_{L_1}]-[E_{L_2}]
\end{equation*} 
admits a cscK representative. For us however it is easier to work with a class
\begin{equation*}
[\omega_{Z_1}] = \pi^*[\omega_{T_{\nabla}}] - \varepsilon^2_1([E_{L_1}]+[E_{L_2}]),0 < \varepsilon^2_1 \ll 1. 
\end{equation*} 
Seyeddali and Sz\'{e}kelyhidi \cite{GaborSeyyedali} prove that the blowup of an extremal metric along a submanifold of codimension greater than $2$ which is invariant under a maximal torus in the Hamiltonian isometry group is still extremal, in classes making the exceptional divisor sufficiently small. They also conjecture that the same result holds when the submanifold has codimension $2$. Thus, assuming their conjecture, $[\omega_{Z_1}]$ admits an extremal representative $\omega_{Z_1}$. On the other hand, according to \cite{SektnanTipler_3folds}, Section 5.2, the Futaki character of $Z_1$ vanishes identically on the K\"ahler cone of $Z_1$. Thus the extremal metric $\omega_{Z_1}$ is necessarily cscK.

Similarly, at the second step, we modify the toric model again by blowing up a second $1$-dimensional toric stratum given by the proper transforms $L'_3, L'_4$ of the lines
\begin{equation*}
L_3 := \{z_1 = z_2 =0\} \cup L_4:=\{z_0 = z_3 = 0\}. 
\end{equation*}
Consider first the Chow stability of $L_3$ and $L_4$ as subvarieties of $T_{\nabla} \cong \PP^3$ with respect to a $SL(H^0(\mathcal{O}(1))^{\vee})$ linarisation of action of the reductive group $\Aut_0(Z_1)$ (the special linear condition will correspond to the required normalisation on a moment map). The linearised action is induced by projective linear transformations of the form
\begin{align*}
\left(\begin{matrix*}[l]
A & 0\\
0 & B
\end{matrix*}
\right),\,A, B \in GL(2, \C),\,\det A  B = 1,
\end{align*}
\emph{with respect to the coordinates $z_i$}. We need to check that the orbit of $L_3 + L_4$, as a cycle of points in the Grassmannian, is closed. Equivalently, we can check that the orbit of $(L_3 \cap L_1, L_3 \cap L_2, L_4 \cap L_1, L_4 \cap L_2)$ as a point of $(\PP^1)^4$ with the induced linearised action of the form $(A, B, A, B)$, is closed. We have, in the coordinates $[z_2 : z_3]$ of $L_1$,
\begin{align*}
& L_3 \cap L_1 = [0:1],\,A (0, 1)^T = (a_{12}, a_{22})^T,\\
& L_4 \cap L_1 = [1:0],\,A (1, 0)^T = (a_{11}, a_{12})^T.
\end{align*}  
Similarly, in the coordinates $[z_0 : z_1]$ of $L_2$,
\begin{align*}
& L_3 \cap L_2 = [1:0],\,B (1, 0)^T = (b_{11}, b_{12})^T,\\
& L_4 \cap L_2 = [0:1],\,B (0,1)^T = (b_{12}, b_{22})^T.
\end{align*}  
We can check closure of the orbit along a 1PS. If $(a_{12}, a_{22})^T$ vanishes as $t \to 0$, then either $(a_{11}, a_{12})^T$ blows up, or (by the special linear normalisation) one of $(b_{11}, b_{12})^T$, $(b_{12}, b_{22})^T$ must blow up. A similar argument applies to the vanishing of any other component. So if the 1PS has a strictly positive component, then it also has a strictly negative one, i.e. the orbit of $L_3 + L_4$ is polystable. By a standard continuity argument, the submanifold $L'_3 \cup L'_4 \subset Z_1$ is Chow polystable for the action of $\Aut_0(Z_1)$ linearised on $[\omega_{Z_1}]$ for sufficiently small $\varepsilon_1 > 0$. 

According to Seyeddali-Sz\'{e}kelyhidi \cite{GaborSeyyedali}, at least conjecturally (since the codimension of $L'_3 \cup L'_4$ is $2$, i.e. critical) this implies that the blowup $Z_2 := \Bl_{L'_3 \cup L'_4} Z_1$ is cscK with respect to the classes
\begin{equation*}
[\omega_{Z_2}] = \pi^*[\omega_{Z_1}] - \varepsilon^2_2([E_{L'_3}] + [E_{L'_4}]).  
\end{equation*}
Indeed, following \cite{GaborSeyyedali}, in general the obstruction to finding a cscK metric on the blowup $\Bl_V M$  of a cscK $(M, \omega_M)$ is given by the vanishing condition
\begin{equation*}
\int_{V} \mu\,\omega_{M} = 0.
\end{equation*}
This is proved in \cite{GaborSeyyedali} in the case of codimension greater than $2$, and conjectured there to hold in the codimension $2$ case as well. The map 
\begin{equation*}
V \mapsto \int_{V} \mu\,\omega_{M}
\end{equation*}
is in fact a moment map for the induced action on the relevant Chow variety. Thus, at least conjecturally, the condition is given by the Chow polystability of the cycle $V$.

Going back to our case, the residual automorphisms at the second step are
\begin{equation}\label{Quartic3foldTorus}
\Aut_0(Z_2) = \left(\begin{matrix*}[l]
A & 0\\
0 & B
\end{matrix*}
\right),\,A, B \in GL(2, \C),\,\det A  B = 1,\,A, B\text{ diagonal}.
\end{equation}

At the third step, we start blowing up components of the base locus. Let $C'_0$, $C'_3$ denote the reduced, disjoint curves given by the proper transforms of $C^{\operatorname{red}}_0$, $C^{\operatorname{red}}_3$ on $Z_2$ (they are disjoint since the node $p_3$ lies on $L_4$). We need to check that the submanifold $C'_0 \cup C'_3\subset Z_2$ is Chow polystable with respect to the action of $\Aut_0(Z_2)$ linearised on $[\omega_{Z_2}]$. For $0 < \epsilon_2 \ll \epsilon_1 \ll 1$, it is sufficient to check that $C^{\operatorname{red}}_0 \cup C^{\operatorname{red}}_3 \subset T_{\nabla}$ is polystable with respect to the linearised action \eqref{Quartic3foldTorus}. The action of a 1PS on the Pl\"ucker coordinates is given by the ordered minors of the matrices     
\begin{align*}
\left(\begin{matrix*}[l]
t^{a} & 0 & 0 & 0\\
0 & t^{b} & 0 &0\\
0 & 0 & t^{c} &0\\
0 & 0 &0 &t^{d} 
\end{matrix*}
\right)\left(\begin{matrix*}[l]
1 & 0\\
0 & 1\\
0 & 1\\
0& 1
\end{matrix*}
\right),\,\left(\begin{matrix*}[l]
t^{a} & 0 & 0 & 0\\
0 & t^{b} & 0 &0\\
0 & 0 & t^{c} &0\\
0 & 0 &0 &t^{d} 
\end{matrix*}
\right)\left(\begin{matrix*}[l]
0 & 1\\
0 & 1\\
0 & 1\\
1& 0
\end{matrix*}
\right) 
\end{align*}
and so explicitly by
\begin{align*}
&[t^{a + b} : t^{a + c} : t^{a + d} : 0: 0 : 0],\,[0: 0: t^{a + d}: 0: t^{b + d}:  t^{c + d}].
\end{align*}
By the special linear condition, one of $a + b$, $c + d$ is strictly negative, unless they both vanish, i.e. $b = -a$, $d = -c$. So one of $a+c$, $b+d$ is strictly negative unless $a = -c$, $b= c$, $d = -c$. But in that case $a + d = - 2c$, $b + c = 2c$ so one of the two is strictly negative unless $c = 0$ and the 1PS is trivial. This shows (conditionally on the Seyeddali-Sz\'{e}kelyhidi conjecture) that the blowup $Z_3 := \Bl_{C'_{0} \cup C'_3} Z_2$ is cscK with respect to the classes
\begin{equation*}
[\omega_{Z_3}] = \pi^*[\omega_{Z_2}] - \varepsilon^2_3([E_{C'_0}] + [E_{C'_3}]),   
\end{equation*}
and also that the group $\Aut_0(Z_3)$ is discrete. So from this point we can proceed to resolve the base locus by iterated blowups of reduced components, preserving the cscK condition in classes which make the exceptional divisors sufficiently small, since there are no more (conjectural) obstructions to be checked. 
\subsection{Adiabatic $K$-instability}\label{AdiabaticSecQuartic}
This Section contains the proof of Lemma stated in Remark \ref{3foldsAdiabaticRmk}.

Let us show that the compactified toric LG model $f\!: Z \to \PP^1$ for a quartic $V_4$ constructed in the previous Section (which is $K$-stable with respect to our polarisations) is adiabatically $K$-unstable. This is similar to the case of the LG model for the cubic surface considered in Section \ref{AdiabaticSec}. Namely, we consider the fibre $F$ of $f$ such that its reduction is the proper transform of the hyperplane $\{z_0 + z_1 + z_2 + z_3 = 0\}$. Since the pencil $\mathfrak{d}_{\Delta}$ is given by $|(z_0 + z_1 + z_2 + z_3)^4,  z_0 z_1 z_2 z_3|$, we see that $F$ is a simple normal crossing divisor with multiplicity $\operatorname{ord}_Z(F) = 4$. As $F$ is simple normal crossing and the total space $Z$ is smooth, by definition of the log canonical threshold, we have $\operatorname{lct}_Z(F) = \frac{1}{4}$. Now we can apply the canonical bundle formula, 
\begin{equation*}
K_{Z} \sim_{\mathbb{Q}} f^*(K_{\PP^1} + M + B)
\end{equation*}
where $M$ and $B$ are the moduli and discriminant divisors, respectively. By the construction of $Z$ as a log Calabi-Yau compactified LG model, we also have
\begin{equation*}
K_{Z} \sim_{\mathbb{Q}} f^*(\mathcal{O}_{\PP^1}(-\infty)).
\end{equation*}  
This shows that $\deg(M+B) = 1$. On the other hand, the discriminant divisor is given by
\begin{equation*}
B = \sum_{P \in \PP^1} \big(1- \operatorname{lct}_Z(f^{*}(P))\big)P.
\end{equation*}
According to \cite{Hattori_moduli}, Theorem 1.1 and Corollary A.17, if we write $B = \sum_P a_P P$, then $f\!: Z \to \PP^1$ is adiabatically $K$-unstable if and only if we have
\begin{equation*}
\max_P a_P > \frac{\deg(M+B)}{2}. 
\end{equation*}
In our case $\max_P a_P \geq \frac{3}{4}$ and $\deg(M+B) = 1$.
\section{$V_8$}
Let $V_8$ be a Fano threefold given by a smooth complete intersection of three quadrics in $\PP^6$. $V_8$ admits several distinct Gorenstein toric degenerations. 

\subsection{$T_{\nabla} = \PP^1 \times \PP^1 \times \PP^1$} We consider the case when $V_8$ degenerates to the toric complete intersection $T_{\Delta}$ with reflexive polytope $\Delta = P_{4250}$ in the Kreuzer-Skarke list. In this case, the polar dual $\nabla$ is given by $P_{30}$, namely, we have $T_{\nabla} \cong \PP^1 \times \PP^1 \times \PP^1$. According to \cite{Ilten_weakLGmodels}, the corresponding toric LG model of $(V_8, 0)$ is the Minkowski Laurent polynomial
\begin{align*}
f = \frac{(x+1)^2(y+1)^2(z+1)^2}{x y z}.
\end{align*}
This extends to the anticanonical pencil on $T_{\nabla} \cong \PP[x_0 : x_1] \times \PP[y_0 : y_1] \times \PP[z_0 : z_1]$ given by
\begin{equation*}
\mathfrak{d}_{\Delta} = |(x_0 + x_1)^2(y_0 + y_1)^2 (z_0 + z_1)^2, x_0 x_1 y_0 y_1 z_0 z_1|.
\end{equation*}
The base locus $\Bs(f)$ is the union of the rational curves with multiplicities
\begin{align*}
& C_{x_0} = \{(y_0 + y_1)^2 (z_0 + z_1)^2 = 0\} \subset \{[0: 1]\} \times \PP[y_0 : y_1] \times \PP[z_0 : z_1],\\
& C_{x_1} = \{(y_0 + y_1)^2 (z_0 + z_1)^2 = 0\} \subset \{[1: 0]\} \times \PP[y_0 : y_1] \times \PP[z_0 : z_1],\\
& C_{y_0} = \{(x_0 + x_1)^2 (z_0 + z_1)^2=0\} \subset \PP[x_0 : x_1] \times \{[0:1]\} \times \PP[z_0 : z_1],\\
& C_{y_1} = \{(x_0 + x_1)^2 (z_0 + z_1)^2=0\} \subset \PP[x_0 : x_1] \times \{[1:0]\} \times \PP[z_0 : z_1],\\
& C_{z_0} = \{(x_0 + x_1)^2 (y_0 + y_1)^2=0\}  \subset \PP[x_0 : x_1]  \times \PP[y_0 : y_1] \times \{[0:1]\},\\
& C_{z_1} = \{(x_0 + x_1)^2 (y_0 + y_1)^2=0\}  \subset \PP[x_0 : x_1]  \times \PP[y_0 : y_1] \times \{[1:0]\}. 
\end{align*}
Assuming the conjecture by Seyeddali-Sz\'{e}kelyhidi, a cscK resolution of $\Bs(f)$ can be obtained as follows. Consider the disjoint components of the reduced base locus given by
\begin{align*}
& C_1 = \{[0: 1]\} \times \{[1 : -1]\} \times \PP[z_0 : z_1],\\ 
& C_2 = \{[1: 0]\} \times \{[1 : -1]\} \times \PP[z_0 : z_1],\\ 
& C_3 = \{[1 : -1]\} \times \{[0:1]\} \times \PP[z_0 : z_1].  
\end{align*}
Then, the cycle $C_1 + C_2 + C_3$ is polystable with respect to the linearised action by $(SL(2, \C))^3$. We blow up $C_1 \cup C_2 \cup C_3$ with small K\"ahler parameter, and consider the cycle given by the proper transforms of 
\begin{align*}
& C_4 = \PP[x_0 : x_1] \times \{[1:0]\} \times \{1:-1\},\\  
& C_5 = \PP[x_0 : x_1]  \times \{[1 : -1]\} \times \{[0:1]\},\\  
& C_6 = \PP[x_0 : x_1]  \times \{[1 : -1]\} \times \{[1:0]\}.   
\end{align*}
Again this is polystable with respect to sufficiently small parameters, and after blowing it up the residual automorphisms are discrete.
\subsection{$T_{\nabla} = \Bl_{L_1 \cup L_2} \PP^3$}\label{V8AltDeg}
An alternative Gorenstein degeneration for $V_8$ is given by the toric complete intersection $T_{\Delta}$ with polytope $\Delta = P_{4166}$ (see \cite{Fanosearch}). The polar dual in this case is given by $\nabla = P_{24}$ and so $T_{\nabla}$ is isomorphic to $\Bl_{L_1 \cup L_2} \PP^3$ with $L_1 := \{z_0 = z_1 =0\}$, $L_2:=\{z_2 = z_3 = 0\}$. The LG potential for the degeneration of $V_8$ to $T_{\nabla}$ is given explicitly in \cite{Coates_quantum},
\begin{align*}
& f = x z^2 + 3 y z^2 + 3 x^{-1} y^2 z^2 + x^{-2} y^3 z^2 + 2 x z + 4 y z + 2 x^{-1} y^2 z + x + y + x y^{-1} z + 3 z\\
& + 3 x^{-1} y z + x^{-2}y^2 z + 4 x y^{-1} + 4 x^{-1} y + 3 x y^{-1} z^{-1} + 3 z^{-1} + 2 x y^{-2} z^{-1} + 4 y^{-1} z^{-1} \\
&+ 2 x^{-1} z^{-1} + 3 x y^{-2} z^{-2} + 3 y^{-1} z^{-2} + x y^{-3} z^{-3} + y^{-2} z^{-3}.
\end{align*}
The rational function $f$ extends to a rational pencil on $T_{\nabla}$. The irreducible components $D_i$ of the toric boundary $D$ of $T_{\nabla}$ are dual to the facets $\Delta_i$ of $\Delta$ for $i = 1,\ldots, 5$. The extension of $f$ to $D$ is determined by the sums of monomials of $f$ supported on the facets of $\Delta$. Namely, we have
\begin{align*}
&f_0 = x z^{2} + x^{-2} y^3 z^2 + x y^{-1} z + x^{-2} y^2 z + 3 z + 3 x^{-1} y z + 3 y z^2 + 3x^{-1} y^2 z^2,\\
&f_1 = x + x z^2 + x y^{-3} z^{-3} + x y^{-1} z + 2 x z + 4 x y^{-1} + 2 x y^{-2} z^{-1} + 3 x y^{-1} z^{-1} + 3 x y^{-2} z^{-2},\\
&f_2 = x + y + x z^2 + x^{-2} y^3 z^2 + 2 x^{-1} y^2 z + 4 y z + 2 x z + 3 x^{-1} y^2 z^2 + 3 y z^2,\\
&f_3 = x + y + x y^{-3} z^{-3} + y^{-2} z^{-3} + 3 z^{-1} + 3 y^{-1} z^{-2} + 3 x y^{-1} z^{-1} + 3 x y^{-2} z^{-2},\\
&f_4 = x y^{-3} z^{-3} + y^{-2} z^{-3} + x y^{-1} z +  x^{-2}y^2 z + 2 x y^{-2} z^{-1} + 4 y^{-1} z^{-1} + 2 x^{-1} z^{-1} + 3 z \\&+ 3 x^{-1} y z,\\
&f_5 = y + x^{-2} y^3 z^2 + y^{-2} z^{-3} + x^{-2}y^2 z + 2 x^{-1} y^2 z + 4 x^{-1} y + 2 x^{-1} z^{-1} + 3 z^{-1} + 3 y^{-1} z^{-2}.
\end{align*}
There are two types of boundary components, given by the exceptional divisors $E_{L_1}$, $E_{L_2} \cong \PP^1 \times \PP^1$, respectively by the proper transforms $H_i$ of the toric boundary components $\{z_i = 0\}$ of $\PP^3$, $H_0,\ldots, H_3 \cong \PP^2$. Consider for example the proper transform $H_0$. It is embedded by the anticanonical sections
\begin{align*}
\{ z_1^3 z_2, z_1^2 z_2^2, z_1 z_2^3, z_1^3 z_3, z_1^2 z_2 z_3, z_1 z_2^2 z_3, 
 z_1^2 z_3^2, z_1 z_2 z_3^2, z_1 z_3^3 \} 
 \end{align*} 
and so the corresponding sum of LG monomials can be expressed as
\begin{align*}
&f_{H_0} = z_1^3 z_2 + 2 z_1^2 z_2^2 + z_1 z_2^3 + z_1^3 z_3 + 4 z_1^2 z_2 z_3 + 
 3 z_1 z_2^2 z_3 + 2 z_1^2 z_3^2 + 3 z_1 z_2 z_3^2 + z_1 z_3^3\\   
& = z_1 (z_2 + z_3) (z_1 + z_2 + z_3)^2.
\end{align*}
Similarly, we have
\begin{align*}
&f_{H_1} = z_0 (z_2 + z_3) (z_0 + z_2 + z_3)^2,\\
&f_{H_2} = z_3 (z_0 + z_1) (z_3 + z_0 + z_1)^2,\\
&f_{H_3} = z_2 (z_0 + z_1) (z_2 + z_0 + z_1)^2.
\end{align*}
On the other hand, $E_{L_1}$ is embedded by the anticanonical sections 
\begin{equation*}
\{ z_0^3 z_2, z_0^2 z_1 z_2, z_0 z_1^2 z_2,  z_1^3 z_2, z_0^3 z_3, z_0^2 z_1 z_3, z_0 z_1^2 z_3, z_1^3 z_3 \}
\end{equation*}
(linear in $z_2, z_3$), and so we can write
\begin{align*}
& f_{E_{L_1}} = z_0^3 z_2 + 3 z_0^2 z_1 z_2 + 3 z_0 z_1^2 z_2 + z_1^3 z_2 + z_0^3 z_3 + 
 3 z_0^2 z_1 z_3 + 3 z_0 z_1^2 z_3 + z_1^3 z_3\\
&= (z_0 + z_1)^3 (z_2 + z_3).
\end{align*}
Similarly, $f_{E_{L_2}} = (z_2 + z_3)^3 (z_0 + z_1)$. Given this explicit description of $\Bs(f)$, we can proceed just as in Section \ref{Quartic3foldSec}  to show that the base locus has a cscK resolution.
\section{$V_6$}
Let $V_6$ be a Fano threefold given by a smooth complete intersection of a quadric and a cubic in  $\PP^5$. Then, $V_6$ admits a Gorenstein toric degeneration to $T_{\Delta}$ with $\Delta = P_{4286}$ (see \cite{Fanosearch}). The polar dual is $\nabla = P_{4}$ and so we have $T_{\nabla} \cong \PP^1 \times \PP^2$. According to \cite{Ilten_weakLGmodels}, a corresponding LG potential is given by 
\begin{equation*}
f = \frac{(x+1)^2(y + z + 1)^3}{x y z}.
\end{equation*}
This extends to the anticanonical pencil on $\PP^1 \times \PP^2$ given by
\begin{equation*}
\mathfrak{d}_{\Delta} = |(x_0 + x_1)^2(y_0 + y_1 + y_2)^3 ,  x_0 x_1 y_0 y_1 y_2|.
\end{equation*}
The components of $\Bs(f)$ are given by
\begin{align*}
&C_{x_0} = \{(y_0 + y_1 + y_2)^3 = 0 \} \subset [0:1]\times\PP^2,\\
&C_{x_1} = \{(y_0 + y_1 + y_2)^3 = 0 \} \subset [1:0]\times\PP^2,\\
&C_{y_0} = \{(x_0 + x_1)^2 (y_1 + y_2)^3 = 0\} \subset \PP^1 \times \PP[y_1 : y_2],\\
&C_{y_1} = \{(x_0 + x_1)^2 (y_0 + y_2)^3 = 0\} \subset \PP^1 \times \PP[y_0 : y_2],\\
&C_{y_2} = \{(x_0 + x_1)^2 (y_0 + y_1)^3 = 0\} \subset \PP^1 \times \PP[y_0 : y_1].
\end{align*}
Consider the disjoint components of the base locus given by  
\begin{align*}
C_1 = \{[0:1]\}\times \{y_0 + y_1 + y_2 = 0\},\,C_2 = \{[1:0]\}\times \{y_0 + y_1 + y_2 = 0\}.
\end{align*}
Then, the cycle $C_1 + C_2$ is polystable with respect to the linearised action by $SL(2, \C) \times SL(3, \C)$. We can blow it up with small K\"ahler parameter, and consider the cycle given by the proper transforms of 
\begin{align*}
C_3 = \{[1:-1]\}\times \PP[y_1 : y_2],\,C_4 = \{[1:-1]\}\times \PP[y_0 : y_2],\,C_5 = \{[1:-1]\}\times \PP[y_0 : y_1].
\end{align*}
For sufficiently small parameters, this is polystable with respect to the residual automorphisms, and after blowing it up the automorphisms are discrete. 
\section{$(2, 2)$ divisor in $\PP^2 \times \PP^2$}\label{22Sec}
Let $X$ be a smooth Fano threefold given by a $(2, 2)$ divisor in $\PP^2 \times \PP^2$. Then $X$ admits a Gorenstein toric degeneration to $T_{\Delta}$ where $\Delta$ is the reflexive polytope $P_{3874}$ (see \cite{Fanosearch}). The polar dual $\nabla$ is given by the reflexive polytope $P_{218}$ and so $T_{\nabla}$ is isomorphic to $\PP^1 \times S_6$ where $S_6$ denotes the smooth toric del Pezzo of degree $6$. According to \cite{Coates_quantum} the corresponding LG potential is given by
\begin{align*}
&f = x + y + z + 2 x^{-1} y z + x^{-2} y^2 z + x^{-2} y z^2 + x^{-3} y^2 z^2 + 2 x^2 y^{-1} z^{-1} + 2 x z^{-1} + 2 x y^{-1} \\
&+ 2 x^{-1}y + 2 x^{-1}z + 2 x^{-2}y z + x^3 y^{-2}z^{-2} + x^2 y^{-1}z^{-2} + x^2 y^{-2}z^{-1} + 2 x y^{-1}z^{-1}\\& + z^{-1} + y^{-1} + x^{-1}.
\end{align*}
The sums of LG monomials supported on facets $\Delta_i$, for $i = 1, \ldots, 7$, are given by
\begin{align*}
&f_0 = x^3 y^{-2}z^{-2}+ x^2 y^{-1}z^{-2} + x^2 y^{-2}z^{-1} + z^{-1} + y^{-1} + x^{-1} + 2 x y^{-1}z^{-1}\\
& = \frac{(x+y) (x+z) \left(x^2+y z\right)}{x y^2 z^2},\\
&f_1 = x + z + x^3 y^{-2}z^{-2} + x^2 y^{-2}z^{-1} + 2 x^2 y^{-1} z^{-1} + 2 x y^{-1} = \frac{(x+z) (x+y z)^2}{y^2 z^2},\\
&f_2 = x + y + x^3 y^{-2}z^{-2} +x^2 y^{-1}z^{-2} + 2 x^2 y^{-1} z^{-1} + 2 x z^{-1} = \frac{(x+y) (x+y z)^2}{y^2 z^2},\\
&f_3 = x^{-2} y^2 z + x^{-3} y^2 z^2 + z^{-1} + x^{-1} + 2 x^{-1}y + 2 x^{-2}y z = \frac{(x+z) (x+y z)^2}{x^3 z},
\end{align*}
\begin{align*}
&f_4 = x + y + z + x^{-2} y^2 z + x^{-2} y z^2 + x^{-3} y^2 z^2 + 2 x^{-1} y z = \frac{(x+y) (x+z) \left(x^2+y z\right)}{x^3},\\
&f_5 = z + x^{-2} y z^2 + x^2 y^{-2}z^{-1} + y^{-1} + 2 x y^{-1} + 2 x^{-1}z  = \frac{(x+y z)^2 \left(x^2+y z\right)}{x^2 y^2 z},\\
&f_6 = x^{-2} y z^2 + x^{-3} y^2 z^2 + y^{-1} + x^{-1} + 2 x^{-1}z + 2 x^{-2}y z = \frac{(x+y) (x+y z)^2}{x^3 y},\\
&f_7 = y + x^{-2} y^2 z + x^2 y^{-1}z^{-2} + z^{-1} + 2 x z^{-1} + 2 x^{-1}y = \frac{(x+y z)^2 \left(x^2+y z\right)}{x^2 y z^2}.
\end{align*}
The toric boundary components $D_i$ of $T_{\nabla}$ dual to $\Delta_i$ are of two types: $\PP^1 \times \tilde{D}_j$, where $\tilde{D}_j$ for $j = 1, \ldots, 6$ is a boundary component of $S_6$, respectively $\{q_k\} \times S_6$, where $q_k$, $k = 1, 2$ are distinct torus fixed points. From the above factorisations (which follow from the Minkowski property of $f$) we see that the components of the reduced base locus of $f$ are of three types: $C'_j := \PP^1 \times \{p_j\}$, where $p_j \in \tilde{D}_i$ is a non torus-fixed point, $C''_j := \{r_j\} \times \tilde{D}_j$, where $r_j$ is a non torus-fixed point, and $\{q_k\} \times L$ where $L \subset S_6$ is the proper transform of a non torus-fixed smooth rational curve. 

At the first step we blow up the components $C'_j$, $j = 1, \ldots, 6$. By the computations of Section \ref{Deg6Sec}, the corresponding cycle is polystable for the linearised action of $GL(2, \C) \times \C^*$. The residual linearised automorphisms are $GL(2, \C)$ acting on the first factor of $\PP^1 \times S_6$. 

Secondly we blow up the proper transforms of $\{q_k\} \times L$ for $k = 1, 2$. The corresponding cycle is polystable for the action of $GL(2, \C)$, and the residual linearised automorphisms are isomorphic to $\C^*$ fixing $q_1,\, q_2$.

Finally, we blow up the proper transform of $C''_j$ for any choice of $j = 1, \ldots, 6$. The corresponding cycle is polystable for the action of $\C^*$. After this step the residual automorphisms are discrete. 

\section{$(2, 2, 2)$ double cover of $\PP^2 \times \PP^2$}
Let $X$ be a Fano threefold given by a a double cover of $\PP^1 \times \PP^1 \times \PP^1$ with branch locus a divisor of tridegree $(2,2,2)$. Then, $X$ also admits a Gorenstein toric degeneration to $T_{\Delta}$ for $\Delta = P_{3874}$ (see \cite{Fanosearch}), so as before $T_{\nabla}$ is isomorphic to $\PP^1 \times S_6$. However, the corresponding LG potential is different and given explicitly \cite{Coates_quantum} by
\begin{align*}
&f = x + y + z + 3 x^{-1} y z + x^{-2} y^2 z + x^{-2} y z^2 + x^{-3} y^2 z^2 + 2 x^2 y^{-1}z^{-1} + 2 x z^{-1} + 2 x y^{-1} \\
&+ 2 x^{-1} y + 2 x^{-1}z + 2 x^{-2} y z + x^3 y^{-2}z^{-2} + x^2 y^{-1} z^{-2} + x^2 y^{-2} z^{-1} + 3 x y^{-1} z^{-1}\\
& + z^{-1} + y^{-1} + x^{-1}. 
\end{align*} 
The sums of LG monomials supported on facets $\Delta_i$ coincide with those found in Section \ref{22Sec} except for
\begin{align*}
&f_0 = x^3 y^{-2}z^{-2}+ x^2 y^{-1}z^{-2} + x^2 y^{-2}z^{-1} + z^{-1} + y^{-1} + x^{-1} + 3 x y^{-1}z^{-1}\\
&= \frac{\left(x^2+x y+y z\right) \left(x^2+x z+y z\right)}{x y^2 z^2},
\end{align*}
\begin{align*}
&f_4 = x + y + z + x^{-2} y^2 z + x^{-2} y z^2 + x^{-3} y^2 z^2 + 3 x^{-1} y z = \frac{\left(x^2+x y+y z\right) \left(x^2+x z+y z\right)}{x^3}.
\end{align*}
Given these factorisations, we obtain the same kind of qualitative description of the base locus of $f$ as in Section \ref{22Sec}, and we can proceed as in that case.
\section{$V_{12}$}
Again $V_{12}$, given by a linear section of $OGr(5, 10)$, degenerates to $T_{\Delta}$ for $\Delta = P_{3874}$, with corresponding LG potential
\begin{align*}
&f = x + y + z + 2 x^{-1} y z + x^{-2} y^2 z + x^{-2} y z^2 + x^{-3} y^2 z^2 + 2 x^2 y^{-1} z^{-1} + 2 x z^{-1}\\
& + 2 x y^{-1} + 2 x^{-1}y + 2 x^{-1}z + 2 x{-2} y z + x^3 y^{-2}z^{-2} + x^2 y^{-1}z^{-2} + x^2 y^{-2} z^{-1}\\
& + 3 x y^{-1} z^{-1} + z^{-1} + y^{-1} + x^{-1}
\end{align*}
(see \cite{Fanosearch, Coates_quantum}). The sums of LG monomials supported on facets $\Delta_i$ coincide with those found in Section \ref{22Sec} except for
\begin{align*}
&f_0 = x^3 y^{-2}z^{-2}+ x^2 y^{-1}z^{-2} + x^2 y^{-2}z^{-1} + z^{-1} + y^{-1} + x^{-1} + 3 x y^{-1}z^{-1}\\
&= \frac{\left(x^2+x y+y z\right) \left(x^2+x z+y z\right)}{x y^2 z^2}. 
\end{align*}
So we can proceed to construct cscK classes on the resolution of $\Bs(f)$ as in that case.
\section{$V_{10}$}\label{V10Sec}
According to \cite{Fanosearch}, a Fano threefold $V_{10}$, given by an intersection of $Gr(2, 5)$ with a linear space and a quadric, admits a Gorenstein toric degeneration to $T_{\Delta}$ with reflexive polytope $\Delta = P_{4073}$. The polar dual $\nabla$ is given by $P_{81}$ and so $T_{\nabla}$ is the unique smooth toric Fano threefold of rank $4$ and degree $40$. There is an isomorphism $T_{\nabla} \cong \Bl_F \Bl_{L_1, L_2} \PP^3$, where $L_1$, $L_2$ are disjoint lines and $F$ is a fibre of the projection $E_{L_1} \to L_1$. The linearised action of $\Aut_0(T_{\nabla})$ is induced by projective linear transformations of the form
\begin{align*}
\left(\begin{matrix*}[l]
A & 0\\
0 & B
\end{matrix*}
\right),\,A, B \in GL(2, \C),\,\det A  B = 1,\,B \textrm{ upper triangular},
\end{align*}
with respect to fixed homogeneous coordinates $z_i$ of $\PP^3$ such that $L_1 = \{z_1 = z_2 = 0\}$, $L_2 = \{z_3 = z_4=0\}$. In particular, $\Aut_0(T_{\nabla})$ is not reductive. 

By \cite{Fanosearch}, a toric LG potential for this degeneration is given explicitly by
\begin{align*}
&f = x z^2 + 2 x z + y z + x y^{-1}z^2 + x + y + 3 x y^{-1} z + 3 z + 3 x y^{-1} + 2 x^{-1} y + y^{-1} z\\
& + x y^{-1} z^{-1} + 3 z^{-1} + 2 x^{-1}y z^{-1} + 3 y^{-1} + 2 x^{-1} + 3 y^{-1} z^{-1} + 4 x^{-1}z^{-1} + x^{-2} y z^{-1}\\
& + y^{-1} z^{-2} + 2 x^{-1} z^{-2} + x^{-2} y z^{-2}.
\end{align*}
The corresponding sums of LG monomials supported on facets $\Delta_i$ of $\Delta$ (for $i = 0,\ldots,6$) are given by
\begin{align*}
&f_0 = x z^2 + 2 x z + x y^{-1} z^2 + y^{-1} z + x^{-2} y z^{-1} + 2 x^{-1} + 3 z + 2 x^{-1} y, \\
&f_1 = x + y + x z^2 + y z + 2 x z = (z+1) (x z+x+y),\\
&f_2 = x + x z^2 + x y^{-1} z^{-1} + x y^{-1}z^2 + 3 x y^{-1} + 3 x y^{-1} z + 2 x z = \frac{x (z+1)^2 (y z+z+1)}{y z},\\
&f_3 = y^{-1} z + y^{-1} z^{-2} + x^{-2} y z^{-1} + x^{-2} y z^{-2} + 3 y^{-1} + y^{-1} z + 2 x^{-1} + 4 x^{-1}z^{-1} + 2 x^{-1} z^{-2},\\
&f_4 = x + y + x y^{-1} z^{-1} + y^{-1} z^{-2} +  x^{-2} y z^{-2} + 2 x^{-1} z^{-2} +3 z^{-1} +2 x^{-1}y z^{-1}, \\
&= \frac{(x+y) (x z+1) (x y z+x+y)}{x^2 y z^2},\\
&f_5 = x y^{-1} z^{-1} + x y^{-1}z^2 + y^{-1} z + y^{-1} z^{-2} + 3 y^{-1} z^{-1} + 3 y^{-1} + 3 x y^{-1} + 3 x y^{-1} z\\
& = \frac{(z+1)^3 (x z+1)}{y z^2},\\
&f_6 = y + y z + x^{-2} y z^{-1} + x^{-2} y z^{-2} + 2 x^{-1} y + 2 x^{-1}y z^{-1} = \frac{y (z+1) (x z+1)^2}{x^2 z^2}.
\end{align*}
Comparing with the computations in Section \ref{V8AltDeg}, in the light of the general characterisation of the reduced base locus given in \cite{Przyjalkowski_CYLG}, Lemma 25, we see that the reduced base locus is the transverse union of lines $\{C_i\}$ contained in the proper tranforms $H_0, \ldots, H_3 \cong \PP^2$, fibres $\{F_j\}$ of the proper transform $E'_{L_2} \cong \PP^1 \times \PP^1$, as well as smooth rational curves $\{G_k\}$ contained in the proper transform $E'_{L_2} \cong \PP^1 \times \PP^1$ and in the exceptional divisor $E \cong \PP^1 \times \PP^1$ for the map $T_{\nabla} \cong \Bl_F \Bl_{L_1, L_2} \PP^3 \to \Bl_{L_1, L_2} \PP^3$. Note that since $F$ lies over a torus-fixed point, the components of $\Bs(f)$ contained in $E$ are disjoint from $\{C_i\} \cup \{F_j\}$.

By construction, a choice of compactified toric LG model is given by a blowup 
\begin{equation*}
Z = \Bl_{\{C_i\} \cup \{F_j\}\cup\{G_k\}} \widehat{T}_{\nabla} 
\end{equation*}
in some fixed order, where $\widehat{T}_{\nabla}$ is some toric model obtained by blowing up lines contained in the toric boundary. As in Section \ref{Quartic3foldSec}, it is convenient to replace $T_{\nabla}$ with its model $\widehat{T}_{\nabla}$ obtained by blowing up the proper transforms of the lines
\begin{equation*}
L_3 := \{z_1 = z_2 =0\} \cup L_4:=\{z_0 = z_3 = 0\}. 
\end{equation*}
We abuse notation slightly and keep using the same notation for the proper transforms of the previous components of $\Bs(f)$ in this new model. There will also be some new components contained in $E'_{L_1}$ and $E$; again, as for the $\{G_k\}$, all these components are disjoint from $\{C_i\} \cup \{F_j\}$. We write $\{G'_k\}$ for the set of all these components. We can thus choose a compactified toric LG model given by
\begin{equation*}
Z = \Bl_{\{G_k\}} ( \Bl_{\{C_i\} \cup \{F_j\}} \widehat{T}_{\nabla} )
\end{equation*}
for fixed orderings of $\{C_i\} \cup \{F_j\}$, respectively $\{G_k\}$. In particular, we can proceed as in Section \ref{Quartic3foldSec} by blowing up at the first step two disjoint components $C_{i}$, $C'_{i}$ contained respectively in $H_0$, $H_3$. The Seyeddali-Sz\'{e}kelyhidi conjecture provides cscK classes on $\Bl_{ C_{i}, C'_{i}} \widehat{T}_{\nabla}$ and, since the residual automorphisms after this step are discrete, also on $Z$, for any choice of ordering of the residual components of $\Bs(f)$.
\section{$(1,1,1,1)$ divisor in $(\PP^1)^4$}\label{1111Sec}
Let $X$ be a rank $4$ smooth Fano threefold given by a $(1,1,1,1)$ divisor in $(\PP^1)^4$. Then, according to \cite{Fanosearch}, $X$ admits a Gorenstein toric degeneration to $T_{\Delta}$ with reflevixe polytope $\Delta = P_{1529}$ and Landau-Ginzburg potential
\begin{equation*}
f = x + y + z + x^{-1} y z + x z^{-1} + x y^{-1} + x^{-1} y + x^{-1} z + x y^{-1} z^{-1} + z^{-1} + y^{-1} + x^{-1}.
\end{equation*}
The polar $\nabla$ is the reflexive polytope $P_{2355}$. According to \cite{Graded}, the Gorenstein toric Fanos $T_{\Delta}$ and $T_{\nabla}$ both have isolated terminal singularities. Moreover, $\Delta$ and $\nabla$ are centrally symmetric, so they have vanishing barycenter. In particular $T_{\nabla}$ admits a K\"ahler-Einstein metric $\omega_{\nabla}$ with isolated orbifold points. 

The LG potential $f$ induces an anticanonical pencil $\mathfrak{d}_{\Delta} = |x_1 + \cdots + x_{12}, x_0|$ on $T_{\nabla}$, and the base locus $\Bs(f)$ can be described as follows. The toric boundary components $D_0, \ldots, D_{11}$ of $T_{\nabla}$ are dual to facets of $\Delta$, with
\begin{align*}
& D_0 = \{x_1 x_4 = x_2 x_3\} \subset \PP[x_1 : x_2 : x_3 : x_4],\,f|_{D_0} = x_1 + x_2 + x_3 + x_4,\\
& D_1 = \{x_1 x_9 = x_5 x_6\} \subset\PP[x_1 : x_5 : x_6 : x_9],\,f|_{D_1} = x_1 + x_5 + x_6 + x_9,\\
& D_2 = \PP[x_7: x_8: x_{12}],\,f|_{D_2} = x_7 + x_8 + x_{12},\\
& D_3 = \{x_4 x_{12} = x_7 x_8\} \subset \PP[x_4 : x_7 : x_8 : x_{12}],\,f|_{D_3} = x_4 + x_7 + x_{8} + x_{12},\\
& D_4 = \PP[x_2: x_4: x_{7}],\,f|_{D_4} = x_2 + x_4 + x_{7},\\
& D_5 = \{x_9 x_{12} = x_{10} x_{11}\} \subset \PP[x_9 : x_{10} : x_{11} : x_{12}],\,f|_{D_5} = x_9 + x_{10} + x_{11} + x_{12},\\
& D_6 = \PP[x_3: x_4: x_{8}],\,f|_{D_6} = x_3 + x_4 + x_{8},\\
& D_7 = \{x_2 x_{12} = x_5 x_7\} \subset \PP[x_2 : x_5 : x_7 : x_{12}],\,f|_{D_7} = x_2 + x_5 + x_7 + x_{12},\\
& D_8 = \{x_3 x_{11} = x_6 x_8\} \subset \PP[x_3 : x_6 : x_8 : x_{11}],\,f|_{D_8} = x_3 + x_6 + x_8 + x_{11},\\
& D_9 = \PP[x_6: x_9: x_{11}],\,f|_{D_9} = x_6 + x_9 + x_{11},\\
& D_{10} = \PP[x_8: x_{10}: x_{11}],\,f|_{D_2} = x_8 + x_{10} + x_{11},\\
& D_{11} = \PP[x_8: x_{9}: x_{12}],\,f|_{D_2} = x_8 + x_{9} + x_{12}.
\end{align*}
Thus, each boundary component is isomorphic to $\PP^2$ or $\PP^1 \times \PP^1$, and $\Bs(f)$ is reduced, with irreducible components given by a line in $\PP^2$ or a union of fibres of $\PP^1 \times \PP^1$.

One can check that there are no lattice points contained in the relative interior of facets of $T_{\Delta}$. It follows that the set of Demazure roots of $T_{\nabla}$ is trivial, and in particular that $\Aut_0(T_{\nabla})$ is isomorphic to its maximal torus (see \cite{Nill_reductive}, Sections 3 and 4). Thus, if $K$ is the Hamiltonian isometry group of $\omega_{\nabla}$, with Lie algebra $\mathfrak{k}$, we have $K^{\C} = \Aut_0(T_{\nabla}) \cong (\C^*)^3$, $\mathfrak{k} \cong \C^3$. Moreover, the full automorphism group $\Aut(T_{\nabla})$ is generated by $\Aut_0(T_{\nabla})$ and a finite number of lattice automorphisms of $\Delta$. We can choose the normalised moment map $\mu$ for the action of $\Aut_0(T_{\nabla})$ on $(T_{\nabla}, \omega_{\nabla})$ to be equivariant with respect to $\Aut(T_{\nabla})$. 

Let $\iota \in \Aut(T_{\nabla})$ be the involution of $T_{\nabla}$ induced by the central symmetry of $\nabla$. By the above explicit description,  $\Bs(f)$ can be written as the union of maximal collections of disjoint irreducible components $V_1$, $V_2$ such that $\iota(V_i) = V_i$. Then, $\iota$ acts on $\mathfrak{k}$ as the central reflection, and we have 
\begin{equation*}
\int_{V_1} \mu\,\omega^2_{\nabla} = \int_{V_1} \iota \cdot \mu\,\omega^2_{\nabla} = - \int_{V_1} \mu\,\omega^2_{\nabla}, 
\end{equation*}
so $V_1$ is polystable under the action of $K^{\C}$, and we can apply the Seyeddali-Sz\'{e}kelyhidi conjecture to $\Bl_{V_1} T_{\nabla}$. 

We repeat the previous step with respect to the second maximal collection $V_2$ and apply the Seyeddali-Sz\'{e}kelyhidi conjecture to $\Bl_{V_2} \Bl_{V_1} T_{\nabla}$.  

At this point, the residual automorphisms are discrete, and we may apply the orbifold resolution result of \cite{ArezzoPacard}.

\addcontentsline{toc}{section}{References}
 
\bibliographystyle{abbrv}
 \bibliography{biblio_Kstab}

\noindent SISSA, via Bonomea 265, 34136 Trieste, Italy\\
Institute for Geometry and Physics (IGAP), via Beirut 2, 34151 Trieste, Italy\\
jstoppa@sissa.it    
\end{document}